\documentclass[reqno,11pt]{article}
\usepackage{a4wide}
\usepackage{color,eucal,enumerate,mathrsfs}
\usepackage[normalem]{ulem}
\usepackage{amsmath,amssymb,epsfig,bbm}
\numberwithin{equation}{section}
\usepackage{hyperref}

\usepackage[latin1]{inputenc}

\newcommand{\N}{\mathbb{N}}

\newcommand{\R}{\mathbb{R}}

\newcommand{\mm}{{\mbox{\boldmath$m$}}}

\newcommand{\aalpha}{{\mbox{\boldmath$\alpha$}}}

\newcommand{\ppi}{{\mbox{\boldmath$\pi$}}}

\newcommand{\sfd}{{\sf d}}

\newcommand{\Kliminf}{K\kern-3pt-\kern-2pt\mathop{\rm lim\,inf}\limits}  
\newcommand{\supp}{\mathop{\rm supp}\nolimits}   
\newcommand{\Lip}{\mathop{\rm Lip}\nolimits}          
\renewcommand{\d}{{\mathrm d}}

\newcommand{\restr}[1]{\lower3pt\hbox{$|_{#1}$}}

\newcommand{\eps}{\varepsilon}  
\newcommand{\nchi}{{\raise.3ex\hbox{$\chi$}}}

\newcommand{\limi}{\varliminf}
\newcommand{\lims}{\varlimsup}

\newcommand{\gopt}{{\rm{OptGeo}}}                   
\newcommand{\prob}[1]{\mathscr P(#1)}                   
\newcommand{\probt}[1]{\mathscr P_2(#1)}                   
\newcommand{\e}{{\rm{e}}}                           

\renewcommand{\mm}{\mathfrak m}                                

\newcommand{\weakgrad}[1]{|\nabla #1|_w} 

\newcommand{\bd}{{\mathbf\Delta}}
\newcommand{\rp}{{\rm p}}
 
\newcommand{\s}{{\rm S}}

\newcommand{\test}[1]{{\rm Test}(#1)}

\newenvironment{proof}{\removelastskip\par\medskip   
\noindent{\em proof} \rm}{\penalty-20\null\hfill$\square$\par\medbreak}

\newtheorem{theorem}{Theorem}[section]

\newtheorem{corollary}[theorem]{Corollary}
\newtheorem{lemma}[theorem]{Lemma}
\newtheorem{proposition}[theorem]{Proposition}

\newtheorem{definition}[theorem]{Definition}

\renewcommand{\u}{\mathcal U}
\renewcommand{\weakgrad}[1]{|D #1|_w}
\newcommand{\beq}{\begin{equation}}
\newcommand{\eeq}{\end{equation}}

\title{The Abresch-Gromoll inequality in a non-smooth setting}
\author{Nicola Gigli, Sunra Mosconi}

\begin{document}
\maketitle
\begin{abstract}
We prove that the Abresch-Gromoll inequality holds on infinitesimally Hilbertian $CD(K,N)$ spaces in the same form as the one available on smooth Riemannian manifolds.
\end{abstract}
\tableofcontents
\section*{Introduction}
The Abresch-Gromoll inequality, proved in \cite{AbreschGromoll90}, is an inequality about thin triangles on Riemannian manifolds with Ricci curvature bounded from below. It can be stated as follows. For $K\leq 0$ and $N>1$ there exists a function $f_{K,N}:(\R^+)^2\to\R^+$ - explicitly given - such that  the following is true.

Given a Riemannian manifold $M$ with Ricci curvature bounded from below by $K\leq 0$ and dimension bounded above by $N$ and a minimizing geodesic $[0,1]\ni t\mapsto \gamma_t\in M$ it holds
\begin{equation}
\label{eq:AG}
E(x)\leq f_{K,N}(h(x), l(x)),\qquad\textrm{ provided }\quad h(x)< l(x)
\end{equation}
where $E,h,l:M\to\R^+$ are defined as
\[
\begin{split}
E(x)&:=\sfd(x,\gamma_0)+\sfd(x,\gamma_1)-\sfd(\gamma_0,\gamma_1),\\
h(x)&:=\inf_{t\in[0,1]}\sfd(x,\gamma_t),\\
l(x)&:=\min\{\sfd(x,\gamma_0),\sfd(x,\gamma_1)\},
\end{split}
\]
and  $\sfd$ is the Riemannian distance on $M$.

Notice that on arbitrary geodesic spaces $(X,\sfd)$, the triangle inequality always ensures that $E(x)\leq 2h(x)$, so that the interest of the Abresch-Gromoll inequality relies on the explicit expression of $f_{K,N}$ which for arbitrary $K\leq 0$ grants
\[
\lim_{h\downarrow0}\frac{f_{K,N}(h, l)}{h}=0,
\]
and in particular for $K=0$ ensures
\[
\lim_{h/l\downarrow 0}\frac{f_{0,N}(h, l)}{h}=0.
\]
Another interesting feature of \eqref{eq:AG} is that while typically Ricci curvature bounds produce `average' estimates, the bound in \eqref{eq:AG} is pointwise in nature.

From this latter fact, it also directly follows that \eqref{eq:AG} holds on Gromov-Hausdorff limits of Riemannian manifolds with Ricci curvature bounded from below by $K$ and dimension bounded from above by $N$.

\bigskip

In recent years, Sturm \cite{Sturm06I,Sturm06II} on one side and Lott-Villani \cite{Lott-Villani09} on the other proposed a synthetic definition of Ricci curvature bounds on metric measure spaces, thus giving a meaning to the statement `the space $(X,\sfd,\mm)$ has Ricci curvature bounded from below by $K$ and dimension bounded from above by $N$', in short: $(X,\sfd,\mm)$ is a $CD(K,N)$ space. The crucial properties of such definition are the consistency with the Riemannian case and the stability w.r.t. measured-Gromov-Hausdorff (in short, mGH) convergence. 

It has been soon realized that the class of $CD(K,N)$ metric measure spaces also includes objects that are not Riemannian in nature (beside obviously the fact that it includes non-smooth structures): a result by Cordero-Erausquin, Sturm and Villani (see the last theorem in \cite{Villani09}) ensures that if we endow $\R^d$ with the Lebesgue measure  and the distance coming from a norm, we always obtain a $CD(0,d)$ space, regardless of the choice of the norm. 

It is easy to check that the qualitative behaviour of inequality \eqref{eq:AG} fails on $\R^d$ equipped with the $L^\infty$ norm, so that we  explicitly get an example of the different geometry arising in $CD(K,N)$ spaces w.r.t. that of  Riemannian manifolds with lower Ricci bounds. This  and related examples motivate the search of a synthetic notion of Ricci curvature bound which, while retaining the stability property w.r.t. mGH convergence, still ensures a `Riemannian-like' behavior of the spaces. A proposal in this direction has been made in \cite{Ambrosio-Gigli-Savare11bis} specifically for what concerns curvature bounds (i.e. no upper bound on the dimension):  according to the slightly finer analysis done in \cite{AGMRS12}, one says that $(X,\sfd,\mm)$ has Riemannian Ricci curvature bounded from below by $K$, (an $RCD(K,\infty)$ space, in short), provided it is a $CD(K,\infty)$ space and $W^{1,2}(X,\sfd,\mm)$ is Hilbert. A few comments about this definition are in order:
\begin{itemize}
\item  In abstract metric measure spaces $W^{1,2}$ is always a Banach space, and in the smooth situation a Finsler manifold is Riemannian if and only if $W^{1,2}$ is Hilbert. In this sense the additional requirement that such a space is Hilbert can be seen as the  non-smooth analogous of `the norm comes from a scalar product' which distinguishes Riemannian manifolds among Finsler ones.
\item  It is not trivial that the $RCD(K,\infty)$ condition is stable w.r.t. mGH convergence. The crucial ingredient to prove it is the identification of the gradient flow in $L^2$ of the (generically non-quadratic) Dirichlet energy and the gradient flow in $(\probt X,W_2)$ of the relative entropy (\cite{Ambrosio-Gigli-Savare11}, see also \cite{GigliKuwadaOhta10} for  the original argument given in the technically simpler case of finite dimensional Alexandrov spaces with curvature bounded from below). Once such identification is obtained, one notices that $W^{1,2}$ is Hilbert if and only if the gradient flow of the `Dirichlet energy' is linear and then read this fact from the optimal transport point of view, where the stability of the gradient flow of the entropy - and thus of the linearity requirement - can be much more easily proved (\cite{Ambrosio-Gigli-Savare11bis}, \cite{Gigli10}, \cite{AGMS12}).
\item The requirement `$W^{1,2}$ is Hilbert' is analytic in nature, not geometric. In particular, as of today, it is not clear whether under an additional curvature assumptions it implies that the tangent spaces - intended either as spaces of directions or as pointed-mGH limits of rescaled spaces - are Euclidean or not. This is unknown even assuming some finite dimensionality condition.
\end{itemize}
According to \cite{Gigli12}, a space $(X,\sfd,\mm)$ such that $W^{1,2}(X,\sfd,\mm)$ is Hilbert will be called infinitesimally Hilbertian.  Aim of this paper is to derive a first concrete geometric consequence out of the infinitesimal Hilbertianity hypothesis. Specifically, we will prove that on infinitesimally Hilbertian $CD(K,N)$ spaces the Abresch-Gromoll inequality \eqref{eq:AG} holds with the same functions $f_{N,K}$ as those given in the smooth setting.

Looking at the original proof of \eqref{eq:AG}, one sees that three ingredients play a role:
\begin{itemize}
\item[1)] The Laplacian comparison estimate for the distance function.
\item[2)] The linearity of the Laplacian.
\item[3)] The weak maximum principle. 
\end{itemize}
Hence the proof can be achieved in the non-smooth setting if we are able to derive the same three ingredients. It is  immediate to see that (2) above is equivalent to the infinitesimal Hilbertianity assumption  (recall that on Finsler manifolds the Laplacian is non-linear in general, and it is so if and only if the manifold is Riemannian, see \cite{Shen98}). The sharp Laplacian comparison estimate has been proved in \cite{Gigli12} - where a  definition of distributional Laplacian has also been proposed - for a large class of $CD(K,N)$ spaces including the infinitesimally Hilbertian ones. Finally, the weak maximum principle for local sub/superminimizers of the energy is a direct consequence of the Poincar\'e inequality (known to be valid on $CD(K,N)$ spaces \cite{Lott-Villani-Poincare}, \cite{Rajala11}), and it is also easily seen that local sub/superminimizers  of the energy can be characterized as those functions having positive/negative distributional Laplacian according to the definition given in \cite{Gigli12}.

Hence indeed all the necessary ingredients are at our disposal, and in this paper we collect them and prove \eqref{eq:AG} in the non-smooth world. In this sense, possibly this work is not so exciting, meaning that we will follow step by step the original proof of \eqref{eq:AG}. The emphasis is instead on the fact that the same computations done originally in \cite{AbreschGromoll90} are actually doable even without an underlying smooth structure, and on the possibility of deriving  a genuinely geometric consequence out of the analytic assumption of infinitesimal Hilbertianity.

\bigskip

The paper aims to give a detailed overview of the analytic tools needed to get (1), (2) and (3) above, so that it can be read without referring to the longer exposition given in \cite{Gigli12} (and \cite{Ambrosio-Gigli-Savare11bis}, \cite{GM12}). These tools are recalled in Section \ref{se:lapcomp}, where all the definitions and theorems are formulated on infinitesimally Hilbertian spaces. Notice also that the definition of distributional Laplacian is only given for locally Lipschitz functions, which is sufficient if one aims at the Laplacian comparison for the distance. This will greatly simplify the exposition. In Section \ref{se:exest} we then prove the Abresch-Gromoll inequality.

\section{Preliminaries}
\subsection{Metric spaces and quadratic transportation distance}
Throughout all the paper $(X,\sfd)$ will be a complete and separable metric space. $(X,\sfd)$ is said to be proper whenever all closed and bounded sets are compact. We denote by $B_r(x)$ the open ball of center $x \in X$ and radius $r>0$, by $\overline B_r(x):=\{y:\sfd(y,x)\leq r\}$ the closed one and by $S_r(x):=\{y:\sfd(y,x)=r\}$ the sphere.

$C([0,1],X)$ is the complete and separable metric  space of continuous curves from $[0,1]$ with values in $X$ equipped with the $\sup$ norm. 

A curve $\gamma \in C([0,1], X)$ is said to be absolutely continuous if there exists a function $f\in L^1([0,1])$ such that
\begin{equation}\label{def:ACcurve}
\sfd(\gamma_t,\gamma_s)\leq \int_s^t f(r) \;\d r, \quad \forall t,s \in [0,1],\ \  t<s.
\end{equation} 
The set of absolutely continuous curves from $[0,1]$ to $X$ will be denoted by $AC([0,1],X)$. More generally, if the function $f$ in \eqref{def:ACcurve} belongs to $L^q([0,1])$, $q \in [1,\infty]$, $\gamma$ is said to be $q$-absolutely continuous, and $AC^q([0,1],X)$ is the corresponding set of $q$-absolutely continuous curves. Recall (see for example Theorem 1.1.2 in \cite{Ambrosio-Gigli-Savare08}) that if $\gamma \in AC^q([0,1],X)$ then the limit
$$\lim_{h \to 0} \frac{\sfd(\gamma_{t+h},\gamma_t)}{|h|} $$
exists for a.e. $t \in [0,1]$. Such function is called \emph{metric speed} or \emph{metric derivative}, is denoted by $|\dot{\gamma}_t|$ and is the minimal (in the a.e. sense) $L^q$ function which can be chosen as $f$ in the right hand side of \eqref{def:ACcurve}. 

For every $t \in [0,1]$, we define the \emph{evaluation map} $\e_t:C([0,1],X)\to X$ as
$$\e_t(\gamma):=\gamma_t, \qquad\qquad \forall \gamma \in C([0,1],X). $$

For $f:X\to \R$ the \emph{local Lipschitz constant} $|D f|:X\to [0,\infty]$ is defined by 
\begin{equation}
|Df|(x):=\limsup_{y\to x} \frac{|f(y)-f(x)|}{\sfd(y,x)}, \quad \text{ if $x$ is not isolated, $0$ otherwise. }
\end{equation}

Given a Borel measure $\sigma$ on $X$,  $\supp(\sigma)$ is the support of $\sigma$, i.e. the smallest closed set on which $\sigma$ is concentrated.

We denote with $\prob X$ the set of Borel probability measures on $X$ and by $\probt X\subset \prob X$ the set of probability measures with finite second moment, i.e. the set of those $\mu\in\prob X$ such that $\int\sfd^2(x,x_0)\,\d\mu(x)<\infty$ for some - and thus any - $x_0\in X$.

For $\mu,\nu\in\probt X$ the quadratic optimal transport distance $W_2$ between them is defined by
\[
W_2^2(\mu,\nu):=\inf\int\sfd^2(x,y)\,\d\aalpha(x,y),
\]
where the infimum is taken among all plans $\aalpha\in\prob{X^2}$ such that $\pi^1_\sharp\aalpha=\mu$ and $\pi^2_\sharp\aalpha=\nu$, $\pi^1,\pi^2$ being the projections on the first and second coordinates respectively.

$(X,\sfd)$ is said to be geodesic if for any $x,y\in X$ there exists a curve $\gamma:[0,1]\to X$, called constant speed geodesic or simply geodesic, such that $\gamma_0=x$, $\gamma_1=y$ and
\[
\sfd(\gamma_t,\gamma_s)=|t-s|\sfd(\gamma_0,\gamma_1),\qquad\forall t,s\in[0,1].
\]
If $(X,\sfd)$ is geodesic, the distance $W_2$ can also be expressed as
\begin{equation}
\label{eq:w2geod}
W_2^2(\mu,\nu)=\inf\iint_0^1|\dot\gamma_t|^2\,\d t\,\d\ppi(\gamma), 
\end{equation}
the infimum being taken among all plans $\ppi\in\prob{C([0,1],X)}$, where the right-hand-side is taken $+\infty$ by definition if $\ppi$ is not concentrated on $AC^2([0,1],X)$. It turns out that the $\inf$ in \eqref{eq:w2geod} is always realized and that any minimizer $\ppi$ is concentrated on geodesics (see e.g. Section 2.2 in \cite{Ambrosio-Gigli11}). Minimizers in \eqref{eq:w2geod} are called optimal geodesic plans and the set of all such minimizers will be denoted by $\gopt(\mu,\nu)$.

\subsection{Sobolev functions}
Here we recall the definition of Sobolev function on a metric measure space $(X,\sfd,\mm)$ with values in $\R$. There are various approaches to such definition, most of them being equivalent, see \cite{Ambrosio-Gigli-Savare-pq}. Here we adopt the definition proposed in \cite{Ambrosio-Gigli-Savare11}, \cite{Ambrosio-Gigli-Savare-pq} along with the presentation given in \cite{Gigli12}.

In this section, $(X,\sfd,\mm)$ is such that
\begin{equation}
\label{eq:mms}
\text{ $(X,\sfd)$ is complete and separable and $\mm$ is a non-negative Radon measure on $X$.}
\end{equation}

\begin{definition}[Test plans]
Let $(X,\sfd,\mm)$ be a metric measure space as in \eqref{eq:mms}. We say that a  plan $\ppi\in\prob{C([0,1],X)}$ is a test plan provided 
\begin{equation}
\label{eq:deftest}
\begin{split}
(\e_t)_\sharp\ppi&\leq C\mm,\qquad \forall t\in[0,1],\ \textrm{ for some constant }C>0,\\
\iint_0^1|\dot\gamma_t|^2\,\d t\,\d\ppi(\gamma)&<\infty.
\end{split}
\end{equation}
\end{definition}

\begin{definition}[Sobolev class and weak upper gradients]
Let $(X,\sfd,\mm)$ be a metric measure space as in \eqref{eq:mms} and  $f:X\to\R$ a Borel function. We say that $f$ belongs to the Sobolev class $\s^2_{loc}(X,\sfd,\mm)$ (resp. $\s^2(X,\sfd,\mm)$) provided there exists  $G\in L^2_{loc}(X,\mm)$ (resp. $G\in L^2(X,\mm)$) such that 
\begin{equation}
\label{eq:defsob}
\int|f(\gamma_1)-f(\gamma_0)|\,\d\ppi(\gamma)\leq \iint_0^1G(\gamma_t)|\dot\gamma_t|\,\d t\,\d\ppi(\gamma),\qquad\forall \textrm{ test plan }\ppi.
\end{equation}
Any such  $G$ is called weak upper gradient of $f$.
\end{definition}
It turns out that for $f\in\s^2_{loc}(X,\sfd,\mm)$ there exists a minimal - in the $\mm$-a.e. sense - weak upper gradient $G$, which we  will denote by $\weakgrad f$. 
A crucial property of Sobolev functions and minimal weak upper gradients is the following locality principle (see for instance Proposition 4.8 in \cite{Ambrosio-Gigli-Savare08} or Corollary 2.21 in \cite{BjornBjorn11}):
\begin{equation}
\label{eq:localweak}
\weakgrad f=\weakgrad g,\qquad\qquad\mm\textrm{-a.e. on }\{f=g\},
\end{equation}
and that for any negligible $N\subset\R$, it holds
\begin{equation}
\label{eq:nullset}
\weakgrad f=0,\qquad\qquad\mm\textrm{-a.e. on }f^{-1}(N).
\end{equation}
These properties allow to define the Sobolev class $\s^2(\Omega)$ of functions defined in an open set $\Omega$ with Sobolev regularity.
\begin{definition}[$\s^2_{loc}(\Omega)$ and $\s^2(\Omega)$]
Let $\Omega\subset X$ be open. The class $\s^2_{loc}(\Omega)$ is the class of all Borel functions $f:\Omega\to \R$ such that $f\nchi\in\s^2(X,\sfd,\mm)$ for any Lipschitz function $\nchi$ with support in $\Omega$. For $f\in\s^2_{loc}(\Omega)$, the minimal weak upper gradient $\weakgrad f:\Omega\to\R^+$ is defined as
\begin{equation}
\label{eq:localomega}
\weakgrad f:=\weakgrad{(f\nchi)},\qquad\mm\text{\rm-a.e. on }\{\nchi=1\},
\end{equation}
where $\nchi:X\to\R$ is any Lipschitz function supported in $\Omega$. Notice that, thanks to \eqref{eq:localweak}, \eqref{eq:localomega} well defines $\mm$-a.e. a map $\weakgrad f\in L^2_{loc}(\Omega)$. 

The subclass $\s^2(\Omega)\subset\s^2_{loc}(\Omega)$ is the set of functions $f$ such that $\weakgrad f\in L^2(\Omega,\mm)$.
\end{definition}

Notice that if $f:X\to\R$ is Lipschitz, then certainly the local Lipschitz constant  is a weak upper gradient, so we get the inequality
\begin{equation}
\label{eq:weaklip}
\weakgrad f\leq|Df|\leq  \Lip(f),\qquad\mm\text{\rm-a.e.},
\end{equation}
where by $\Lip(f)$ we mean the (global) Lipschitz constant of $f$.

Basic calculus rules are
\begin{equation}
\label{eq:calcrule}
\begin{split}
\weakgrad{(\alpha f+\beta g)}&\leq |\alpha|\weakgrad f+|\beta|\weakgrad g,\qquad\forall f,g\in\s^2_{loc}(\Omega),\ \alpha,\beta\in\R,\\
\weakgrad{(fg)}&\leq |f|\weakgrad g+|g|\weakgrad f,\qquad\forall f,g\in\s^2_{loc}(\Omega)\cap L^\infty_{loc}(\Omega),
\end{split}
\end{equation}
these inequalities being valid $\mm$-a.e.. The second one in \eqref{eq:calcrule} also holds for $f\in\s^2_{loc}(\Omega)$ and $g$ Lipschitz, meaning that also in this case the product $fg$ belongs to $\s^2_{loc}(\Omega)$ and the weak Leibniz rule holds. The following version of the chain rule is also available:
\begin{equation}
\label{eq:chainbase}
\weakgrad{(\varphi\circ f)}=|\varphi'|\circ f\,\weakgrad f,\qquad\mm\textrm{-a.e.,}
\end{equation}
for $f\in\s^2_{loc}(\Omega)$ and $\varphi:I\to \R$ Lipschitz, where $I\subset\R$ is any interval such that $\mm(f^{-1}(\R\setminus I))=0$ (in \eqref{eq:chainbase}, thanks to \eqref{eq:nullset}, we can and will define the right hand side as 0 at points $x$ such that $\varphi$ is not differentiable at $f(x)$).

The Sobolev space $W^{1,2}(X,\sfd,\mm)$ is then defined as $L^2(X,\mm)\cap \s^2(X,\sfd,\mm)$ endowed with the norm
\[
\|f\|_{W^{1,2}}^2:=\|f\|^2_{L^2}+\|\weakgrad f\|^2_{L^2}.
\]
Notice that in general $W^{1,2}(X,\sfd,\mm)$ is not an Hilbert space (consider for instance the case of finite dimensional Banach spaces). Spaces such that $W^{1,2}$ is Hilbert will be called infinitesimally Hilbertian, see Section \ref{se:infhil}

\subsection{Curvature-Dimension bounds}
Here we recall the definition of $CD(K,N)$, $1<N<\infty$, spaces given by Sturm in \cite{Sturm06II} and Lott-Villani in \cite{Lott-Villani09} (the latter reference deals with the case $K=0$ only) and their basic properties.

Let $u:[0,\infty)\to\R$ be a convex continuous and sublinear (i.e. $\lim_{z\to+\infty}\frac{u(z)}{z}=0$) function satisfying $u(0)=0$. Let $\mathcal M^+(X)$ be the space of finite non-negative Borel measures on $X$. The \emph{internal energy} functional  $\u:\mathcal M^+(X)\to\R\cup\{+\infty\}$ associated to $u$ is well defined by the formula
\[
\u (\mu):=\int u(\rho)\,\d\mm,\qquad \mu=\rho\mm+\mu^s,\ \mu^s\perp\mm.
\]
Jensen's inequality ensures that if $\mm(\supp(\mu))<\infty$, then $\u(\mu)>-\infty$. More generally, the functional $\u$ is lower semicontinuous in $\prob{B}\subset\mathcal M^+ (X)$ w.r.t. convergence in duality with $C_b(B)$, for any closed set $B$ such that $\mm(B)<\infty$.

Functions $u$ of interest for us are 
\[
\begin{split}
u_N(z)&:=-z^{1-\frac1N},\qquad N\in(1,\infty),\\
\end{split}
\]
and we will denote the associated internal energies by $\u_N$ respectively.

For $N\in(1,\infty)$ and $K\in\R$, the \emph{distorsion coefficients} $\tau_{K,N}^{(t)}(\theta)$ are the functions $[0,1]\times[0,\infty)\ni (t,\theta)\mapsto \tau_{K,N}^{(t)}(\theta)\in[0, +\infty]$ defined by
\[
\tau^{(t)}_{K,N}(\theta):=\left\{
\begin{array}{ll}
+\infty,&\qquad\textrm{if }K\theta^2\geq (N-1)\pi^2,\\
\displaystyle{t^{\frac1N}\left(\frac{\sin(t\theta\sqrt{K/(N-1)})}{\sin(\theta\sqrt{K/(N-1)})}\right)^{1-\frac1N}},&\qquad\textrm{if }0<K\theta^2<(N-1)\pi^2,\\
t,&\qquad\textrm{if }K\theta^2=0,\\
\displaystyle {t^{\frac1N}\left(\frac{\sinh(t\theta\sqrt{-K/(N-1)})}{\sinh(\theta\sqrt{-K/(N-1)})}\right)^{1-\frac1N}},&\qquad\textrm{if }K\theta^2<0.
\end{array}
\right.
\]

\begin{definition}[Weak Ricci curvature bound] Let $(X,\sfd,\mm)$ be a  metric measure space such that bounded sets have finite $\mm$-measure. 
We say that $(X,\sfd,\mm)$ is a $CD(K,N)$ space, $K\in\R$, $N\in(1,\infty)$ provided for any $\mu,\nu\in\prob{\supp(\mm)}$ with bounded support there exists $\ppi\in\gopt(\mu,\nu)$ such that
\begin{equation}
\label{eq:cd}
\u_{N'}((\e_t)_\sharp\ppi)\leq-\int\tau^{(1-t)}_{K,N'}\big(\sfd(\gamma_0,\gamma_1)\big)\rho^{-\frac1{N'}}(\gamma_0)+\tau^{(t)}_{K,N'}\big(\sfd(\gamma_0,\gamma_1)\big)\eta^{-\frac1{N'}}(\gamma_1)\,\d\ppi(\gamma),\quad\forall t\in[0,1],
\end{equation}
for any $N'\geq N$, where $\mu=\rho\mm+\mu^s$ and $\nu=\eta\mm+\nu^s$, with $\mu^s,\nu^s\perp\mm$.
\end{definition}
In the following proposition we collect those basic properties of $CD(K,N)$ spaces we will use later on. 

Recall that $(X,\sfd,\mm)$ is said to be doubling  provided there exists a constant $C>0$ such that
\begin{equation}
\label{eq:doubling}
\mm(B_{2r}(x))\leq C\mm(B_r(x)),\qquad\forall r>0,\ x\in X.
\end{equation}
\begin{proposition}[Basic properties of $CD(K,N)$ spaces]\label{prop:basecd}
Let $(X,\sfd,\mm)$ be a $CD(K,N)$ space with $\supp(\mm)=X$, $K\in\R$, $N\in (1,+\infty)$. Then $(X,\sfd)$ is proper and geodesic, $\mm$ is doubling and $(X,\sfd,\mm)$ supports a weak local 1-1 Poincar\'e inequality, i.e. for any bounded Borel function $f:X\to\R$ and any upper gradient $G$ of $f$ it holds
\begin{equation}
\label{eq:11poinc}
\frac{1}{\mm(B_r(x))}\int_{B_r(x)}\Big|f-\langle f\rangle_{B_r(x)} \Big|\,\d\mm\leq Cr\frac{1}{\mm(B_{2r}(x))}\int_{B_{2r}(x)}G\,\d\mm,
\end{equation}
where $\langle f\rangle_{B_r(x)}:=\frac{1}{\mm(B_r(x))}\int_{B_r(x)} f\,\d\mm$, for some constant $C$ depending only on $(X,\sfd,\mm)$. Also, the Bishop-Gromov comparison estimates holds, i.e. for any $x\in \supp(\mm)$ it holds
\begin{equation}
\label{eq:BGv}
\begin{split}
\frac{\mm(B_r(x))}{\mm(B_R(x))}&\geq\left\{
\begin{array}{ll}
\displaystyle{\frac{\int_0^r\sin(t\sqrt{K/(N-1)})^{N-1}\,\d t}{\int_0^R\sin(t\sqrt{K/(N-1)})^{N-1}\,\d t}}&\qquad\textrm{ if }K>0,\\
&\\
\displaystyle{\frac {r^N}{R^N}}&\qquad\textrm{ if }K=0,\\
&\\
\displaystyle{\frac{\int_0^r\sinh(t\sqrt{K/(N-1)})^{N-1}\,\d t}{\int_0^R\sinh(t\sqrt{K/(N-1)})^{N-1}\,\d t}}&\qquad\textrm{ if }K<0,\\
\end{array}
\right.
\end{split}
\end{equation}
for any $0<r\leq R\leq \pi\sqrt{(N-1)/(\max\{K,0\})}$.

Finally, if $K>0$ then $(\supp(\mm),\sfd)$ is compact and with diameter at most $\pi\sqrt{\frac{N-1}K}$.
\end{proposition}
\begin{proof}
For the Poincar\'e inequality see \cite{Lott-Villani-Poincare} for the original argument requiring the non-branching condition and the more recent paper \cite{Rajala11} for the same result without such assumption. For the other properties, see \cite{Sturm06II} or the final chapter of \cite{Villani09}.
\end{proof}
We conclude this introduction recalling that on general metric measure spaces $(X,\sfd,\mm)$, given a Lipschitz function $f:X\to\R$, typically the local Lipschitz constant $|Df|$ - which is a metric object - and the minimal weak upper gradient $\weakgrad f$ - which is metric-measure theoretic - do not coincide, the only information available being $\weakgrad f\leq |D f|$ $\mm$-a.e., which follows from the fact that the local Lipschitz constant is an upper gradient for Lipschitz functions. A deep result due to Cheeger \cite{Cheeger00} ensures that if $\mm$ is doubling and $(X,\sfd,\mm)$ supports a weak local 1-1 Poincar\'e inequality, then it holds
\begin{equation}
\label{eq:cheeger2}
\weakgrad f=|D f|,\qquad \mm-a.e.\qquad \forall f:X\to\R\textrm{ locally Lipschitz}.
\end{equation}
In particular,  we have the  following result.
\begin{theorem}\label{thm:cheeger}
Let $(X,\sfd,\mm)$ be a $CD(K,N)$ space, $K\in \R$, $N\in(1,\infty)$. Then \eqref{eq:cheeger2} holds.
\end{theorem}

\section{Laplacian comparison}\label{se:lapcomp}
\subsection{Infinitesimally Hilbertian spaces and the object $\nabla f\cdot\nabla g$}\label{se:infhil}
On a general metric measure spaces the map $\s^2(X,\sfd,\mm)\ni f\mapsto \int \weakgrad f^2\,\d\mm$ may not be a quadratic form. This can be easily seen if one considers $\R^d$ equipped with the Lebesgue measure and a non-Hilbertian norm.

In this section, as well as in the rest of the paper, we will focus on those metric measure spaces which, from the Sobolev calculus' point of view, resemble a Riemannian structure rather than a general Finsler one. The definition as well as the foregoing discussion comes from \cite{Ambrosio-Gigli-Savare11bis} and \cite{Gigli12}.
\begin{definition}[Infinitesimally Hilbertian spaces]
Let $(X,\sfd,\mm)$ be as in \eqref{eq:mms}. We say that it is infinitesimally Hilbertian whenever the map $\s^2(X, \sfd, \mm)\ni f\mapsto \|\weakgrad f\|_{L^2}^2$ is a quadratic form, i.e. when it satisfies
\begin{equation}
\label{eq:infhil}
\|\weakgrad{(f+g)}\|_{L^2}^2+\|\weakgrad{(f-g)}\|_{L^2}^2=2\Big(\|\weakgrad{f}\|_{L^2}^2+\|\weakgrad{g}\|_{L^2}^2\Big).
\end{equation}
\end{definition}
On infinitesimally Hilbertian spaces and for given Sobolev functions $f,g$  one can define a bilinear object $\nabla f\cdot \nabla g$ which plays the role of  the scalar product of their gradients. This can be done without defining what the gradient of a Sobolev function actually is, as in metric spaces this notion requires more care (see \cite{Gigli12} for more details). Thus, the spirit of the definition is similar to the one that leads to the definition of the carr\'e du champ $\Gamma(f,g)$ in the context of Dirichlet forms. Actually, on infinitesimally Hilbertian spaces the map 
\[
L^2(X,\mm)\cap\s^2(X,\sfd,\mm)\ni f\qquad\mapsto \qquad\int_X\weakgrad f^2\,\d\mm, 
\]
is a regular and strongly local Dirichlet form on $L^2(X,\mm)$, so that the object $\nabla f\cdot\nabla g$ that we are going to define could actually be introduced just as the carr\'e du champ $\Gamma(f,g)$ associated to this Dirichlet form. Despite this, we are going to use a  different definition and a different notation since our structure is richer than the one available when working with  abstract Dirichlet forms, because we have a metric measure space $(X,\sfd,\mm)$ satisfying the assumption \eqref{eq:infhil} and not only a topological space $(X,\tau)$ endowed with a measure $\mm$ and a Dirichlet form $\mathcal E$. This additional feature reflects in several ways in the analysis which we carry out:
\begin{itemize}
\item First and foremost, the proof of the Laplacian comparison estimates is based on tools which are directly related to the geometry of the space (see Lemma \ref{le:horver}) and seems to have nothing to do with semigroup theory.
\item The definition \ref{def:nablafnablag} given below makes sense even on spaces which are not infinitesimally Hilbertian and in this higher generality provides a reasonable definition of what is `the differential of $f$ applied to the gradient of $g$' (see \cite{Gigli12}). In this sense, the approach we propose is more general than the one available in the context of Dirichlet form and formula \eqref{eq:defnablanabla} can be seen as a sort of nonlinear variant of the polarization identity. 
\item Also, the definition of carr\'e du champ requires the introduction of the infinitesimal generator of the form, i.e. the corresponding Laplacian, which at this stage must be interpreted as unbounded operator defined on $L^2$ with values in $L^2$. Yet, in order to state the Laplacian comparison property of the distance function we need to deal with a measure-valued Laplacian. Rather than introducing two different definitions of Laplacian, we first define $\nabla f\cdot \nabla g$ without referring to it and only later speak about measure-valued object.
\end{itemize}
\begin{definition}[The object $\nabla f\cdot\nabla g$]\label{def:nablafnablag}
Let $(X,\sfd,\mm)$ be an infinitesimally Hilbertian space, $\Omega\subseteq X$ an open set and $f,g\in \s^2_{loc}(\Omega)$. The map $\nabla f\cdot\nabla g:\Omega\to\R$ is $\mm$-a.e. defined as
\begin{equation}
\label{eq:defnablanabla}
\nabla f\cdot\nabla g:=\mathop{\rm ess\, inf}_{\eps>0}\frac{\weakgrad{(g+\eps f)}^2-\weakgrad g^2}{2\eps}.
\end{equation}
\end{definition}
Notice that as a direct consequence of the locality stated in \eqref{eq:localweak}, also the object $\nabla f\cdot\nabla g$ is local, i.e.:
\begin{equation}
\label{eq:local}
\nabla f\cdot\nabla g=\nabla \tilde f\cdot\nabla \tilde g,\qquad\mm\text{\rm-a.e.  on}  \ \{f=\tilde f\}\cap\{g=\tilde g\}\cap\Omega.
\end{equation}
In the following theorem we collect the main properties of $\nabla f\cdot\nabla g$, the most relevant and non-trivial ones being its symmetry and the natural chain and Leibniz rules. A  byproduct of these results is that the infinitesimal Hilbertianity condition \eqref{eq:infhil} is equivalent to the localized version 
\[
\weakgrad{(f+g)}^2+\weakgrad{(f-g)}^2=2\Big(\weakgrad f^2+\weakgrad g^2\Big),\qquad\mm\text{\rm-a.e., }\ \forall f,g\in\s^2(X,\sfd,\mm).
\]

\begin{theorem}
Let $(X,\sfd,\mm)$ be infinitesimally Hilbertian and $\Omega\subseteq X$ an open set. 

Then the following hold.
\begin{itemize}
\item\underline{`Cauchy-Schwartz'.} For any $f,g\in\s^2_{loc}(\Omega)$ it holds
\begin{align}
\label{eq:1}
\nabla f\cdot\nabla f&=\weakgrad f^2,\\
\label{eq:boundfg}
\big|\nabla f\cdot\nabla g\big|&\leq \weakgrad f\weakgrad g,
\end{align}
$\mm$-a.e. on $\Omega$.
\item\underline{Linearity in $f$.} For any  $f_1,f_2,g\in \s^2_{loc}(\Omega)$ and $\alpha,\beta\in\R$ it holds
\begin{equation}
\label{eq:linf}
\nabla(\alpha f_1+\beta f_2)\cdot\nabla g=\alpha\nabla f_1\cdot\nabla g+\beta \nabla f_2\cdot\nabla g,\qquad\mm\text{\rm -a.e. on }\Omega.
\end{equation}
\item\underline{Chain rule in $f$.}
Let $f\in\s^2_{loc}(\Omega)$ and $\varphi:\R\to\R$  Lipschitz. Then for any $g\in\s^2_{loc}(\Omega)$ it holds
\begin{equation}
\label{eq:chainf}
\nabla(\varphi\circ f)\cdot\nabla g=\varphi'\circ f\,\nabla f\cdot\nabla g,\qquad\mm\text{\rm-a.e. on }\Omega,
\end{equation}
where the right hand side is taken 0 by definition on those $x\in\Omega$ such that $\varphi$ is not differentiable at $f(x)$. 
\item\underline{Leibniz rule in $f$.}
For $f_1,f_2\in\s^2_{loc}(\Omega)\cap L^{\infty}_{loc}(\Omega)$ and $g\in\s^2_{loc}(\Omega)$ the Leibniz rule
\begin{equation}
\label{eq:leibf}
\nabla (f_1f_2)\cdot\nabla g=f_1\nabla f_2\cdot\nabla g+f_2\nabla f_1\cdot\nabla g,\qquad\mm\text{\rm-a.e. on }\Omega,
\end{equation}
holds. 
\item\underline{Symmetry.} For any $f,g\in\s^2_{loc}(\Omega)$ it holds
\begin{equation}
\label{eq:simm}
\nabla f\cdot\nabla g=\nabla g\cdot\nabla f,\qquad\mm\text{\rm-a.e. on }\Omega.
\end{equation}
In particular, the object $\nabla f\cdot\nabla g$ is also linear in $g$ and satisfies chain and Leibniz rules analogous to those valid for $f$.
\end{itemize}
\end{theorem}
\begin{proof}
All the properties are stated as $\mm$-a.e. equalities on $\Omega$. Hence, by the very definition of $\s^2_{loc}(\Omega)$ and \eqref{eq:local} to conclude it is sufficient to deal with the case of $\Omega=X$ and functions in $\s^2(X,\sfd,\mm)$.

The identity \eqref{eq:1} is a direct consequence of the definition. Taking into account that $\weakgrad{(g+\eps f)}\leq \weakgrad g+\eps\weakgrad f$ for any $\eps>0$, we get
\begin{equation}
\label{eq:lato1}
\nabla f\cdot\nabla g\leq \weakgrad f\weakgrad g,\qquad\mm\textrm{-a.e..}
\end{equation}
From the first inequality in \eqref{eq:calcrule} we get that the map $\s^2(X,\sfd,\mm) \ni f\mapsto\weakgrad{f}$ is $\mm$-a.e. convex, in the sense that
\[
\weakgrad{((1-\lambda)f+\lambda g)}\leq(1-\lambda)\weakgrad f+\lambda\weakgrad g,\qquad\mm\text{\rm-a.e.}\quad \forall f,g\in \s^2(X,\sfd,\mm),\ \lambda\in[0,1].
\]
It follows that $\R\ni\eps\mapsto\weakgrad{(g+\eps f)}$ is also $\mm$-a.e. convex and, being non-negative, also $\R\ni\eps\mapsto\weakgrad{(g+\eps f)}^2/2$ is $\mm$-a.e. convex. In particular, the $\mathop{\rm ess\, inf}_{\eps>0}$ in definition \eqref{eq:defnablanabla} can be substituted with $\lim_{\eps\downarrow0}$ $\mm$-a.e., and thus  we  easily get that for given $g\in \s^2(X,\sfd,\mm)$ we have
\begin{equation}
\label{eq:1conv}
\textrm{the map }\s^2(X,\sfd,\mm)\ni f\quad\mapsto\quad\nabla f\cdot\nabla g\textrm{ is }\mm\textrm{-a.e. positively 1-homogeneous and convex},
\end{equation}
and that
\[
\frac{\weakgrad{(g+\eps f)}^2-\weakgrad g^2}{2\eps}\leq \frac{\weakgrad{(g+\eps' f)}^2-\weakgrad g^2}{2\eps'},\qquad\mm\text{\rm-a.e.}\quad \forall \eps,\eps'\in\R\setminus\{0\},\ \eps\leq \eps',
\]
so that we obtain $\mm$-a.e.
\begin{equation}
\label{eq:nonuguali}
\nabla f\cdot\nabla g=\inf_{\eps>0}\frac{\weakgrad{(g+\eps f)}^2-\weakgrad g^2}{2\eps}\geq\sup_{\eps<0}\frac{\weakgrad{(g+\eps f)}^2-\weakgrad g^2}{2\eps}=-\nabla(- f)\cdot\nabla g.
\end{equation}
Now plug  $\eps f$ in place of $f$ in \eqref{eq:infhil} to get 
\[
\int \frac{\weakgrad{(g+\eps f)}^2-\weakgrad g^2}{2\eps}\,\d\mm=-\int \frac{\weakgrad{(g-\eps f)}^2-\weakgrad g^2}{2\eps}\,\d\mm+\eps\int\weakgrad f^2\,\d\mm.
\]
Letting $\eps\downarrow0$ we obtain $\int \nabla f\cdot\nabla g\,\d\mm=-\int\nabla (-f)\cdot\nabla g\,\d\mm $, which by \eqref{eq:nonuguali} forces 
\begin{equation}
\label{eq:uguali}
\nabla f\cdot\nabla g=-\nabla(-f)\cdot\nabla g,
\end{equation}
which in particular, by \eqref{eq:lato1}, gives \eqref{eq:boundfg}. 

For given $g\in\s^2(X,\sfd,\mm)$, \eqref{eq:1conv} yields that $f\mapsto -\nabla (-f)\cdot \nabla g $ is $\mm$-a.e. positively 1-homogeneous and concave, hence from \eqref{eq:uguali} we deduce the linearity in $f$ of $\nabla f\cdot\nabla g$, i.e. \eqref{eq:linf} is proved.

We now turn to the chain rule in \eqref{eq:chainf}.  Notice that the linearity in $f$ and the inequality \eqref{eq:boundfg} immediately yield
\begin{equation}
\label{eq:1lip}
\big|\nabla f\cdot\nabla g-\nabla \tilde f\cdot\nabla g\big|\leq \weakgrad{(f-\tilde f)}\weakgrad{g}.
\end{equation}
Moreover, thanks to \eqref{eq:linf}, \eqref{eq:chainf} is obvious if  $\varphi$ is linear, and since \eqref{eq:chainf} is unchanged if we add a constant to $\varphi$, it is also true if $\varphi$ is affine. Then, using the locality property \eqref{eq:local}  we also get \eqref{eq:chainf} for $\varphi$ piecewise affine (notice that to deal with the negligible points of non-differentiability of $\varphi$ we are using \eqref{eq:nullset} and the fact that the right hand side of \eqref{eq:chainf} is taken 0 by definition at those $x$'s such that $\varphi$ is not differentiable at $f(x)$). To conclude in the general case, let $\tilde\mm\in\prob X$ be any measure having the  same negligible sets of $\mm$ (use the Lindel\"of property of $(X,\sfd)$ to find such $\tilde\mm$) and observe that the measure $f_\sharp\tilde\mm$ on $\R$ is a Borel probability measure, and hence is Radon. From this fact it is easy to check that for general Lipschitz $\varphi$ we can find a sequence  $(\varphi_n)$ of piecewise affine functions such that $\varphi_n'(z)\to\varphi'(z)$ as $n\to\infty$ for $f_\sharp\tilde\mm$-a.e. $z\in\R$. By the choice of $\tilde\mm$ we then get that $\mm$-a.e. it holds $(\varphi'-\varphi_n')\circ f\to0$ as $n\to\infty$.  Thus using \eqref{eq:1lip} and \eqref{eq:chainbase} we deduce that 
\[
\big|\nabla (\varphi\circ f)\cdot\nabla g-\nabla (\varphi_n\circ f)\cdot\nabla g\big|\leq\weakgrad{((\varphi-\varphi_n)\circ f )}\weakgrad g= |\varphi'-\varphi_n'|\circ f\weakgrad f\weakgrad g\to 0,
\]
$\mm$-a.e., as desired.

The Leibniz rule \eqref{eq:leibf} is a consequence of the chain rule  \eqref{eq:chainf} and the linearity \eqref{eq:linf}: indeed, up to adding a constant to both $f_1$ and $f_2$, we can assume that $\mm$-a.e. it holds $f_1,f_2\geq c$ for some $c>0$, then notice that from \eqref{eq:chainf} and \eqref{eq:linf} we get
\[
\begin{split}
\nabla (f_1f_2)\cdot\nabla g&=f_1f_2\nabla(\log(f_1f_2))\cdot\nabla g=f_1f_2\nabla(\log f_1+\log f_2)\cdot\nabla g\\
&=f_1f_2\big(\nabla(\log f_1)\cdot\nabla g+\nabla(\log f_2)\cdot\nabla g\big)=f_1f_2\left(\frac1{f_1}\nabla f_1\cdot\nabla g+\frac1{f_2}\nabla f_2\cdot\nabla g\right)\\
&=f_2\nabla f_1\cdot\nabla g+f_1\nabla f_2\cdot\nabla g.
\end{split}
\]
To conclude it is now sufficient to show the symmetry relation \eqref{eq:simm}. For this we shall need some auxiliary intermediate results. The first one concerns continuity in $g$ of the map $\s^2(X,\sfd,\mm)\ni g\mapsto\int \nabla f\cdot\nabla g\,\d\mm$. More precisely, we claim that
\begin{equation}
\label{eq:contg}
\begin{split}
&\text{\rm given a sequence $(g_n)\subset \s^2(X,\sfd,\mm)$ and $g\in\s^2(X,\sfd,\mm)$ such that}\\
&\text{\rm $\lim_{n\to\infty}\int\weakgrad{(g_n-g)}^2\,\d\mm=0$, for any $f\in\s^2(X,\sfd,\mm)$ it holds}\\
&\lim_{n\to\infty}\int\nabla f\cdot\nabla g_n\,\d\mm=\int\nabla f\cdot\nabla g\,\d\mm.
\end{split}\end{equation}
To see this, notice that for any $\eps\neq 0$ and under the same assumptions it holds
\[
\lim_{n\to\infty}\int \frac{\weakgrad{(g_n+\eps f)}^2-\weakgrad{g_n}^2}{\eps}\,\d\mm=\int \frac{\weakgrad{(g+\eps f)}^2-\weakgrad{g}^2}{\eps}\,\d\mm.
\]
Now recall that $\R^+\ni\eps\mapsto\frac{\weakgrad{(g_n+\eps f)}^2-\weakgrad{g_n}^2}{\eps}$ is $\mm$-a.e. increasing and converges to $\nabla f\cdot\nabla g_n$ $\mm$-a.e. to get 
\[
\lims_{n\to\infty}\int\nabla f\cdot\nabla g_n\,\d\mm\leq\lim_{n\to\infty}\int \frac{\weakgrad{(g_n+\eps f)}^2-\weakgrad{g_n}^2}{\eps}\,\d\mm=\int \frac{\weakgrad{(g+\eps f)}^2-\weakgrad{g}^2}{\eps}\,\d\mm,
\]
and eventually passing to the limit as $\eps\downarrow 0$ we deduce
\begin{equation}
\label{eq:limsg}
\lims_{n\to\infty}\int\nabla f\cdot\nabla g_n\,\d\mm\leq\int\nabla f\cdot\nabla g\,\d\mm.
\end{equation}
To get the $\limi$ inequality, just  notice that directly from the definition and the linearity in $f$ expressed in \eqref{eq:linf} we obtain $\nabla f\cdot\nabla(-g_n)=-\nabla f\cdot\nabla g_n$. Hence applying \eqref{eq:limsg} with $g_n,g$ replaced by $-g_n,-g$ gives \eqref{eq:contg}.

We shall use \eqref{eq:contg} to obtain an integrated chain rule for $g$, i.e.:
\begin{equation}
\label{eq:chaingint}
\int \varphi'\circ g\nabla f\cdot\nabla g\,\d\mm=\int\nabla f\cdot\nabla(\varphi\circ g)\,\d\mm.
\end{equation}
To get this, start observing that letting $\eps\downarrow 0$ in the trivial identity
\[
\frac{\weakgrad{(\alpha g+\eps f)}^2-\weakgrad{(\alpha g)}^2}{2\eps}=\alpha\frac{\weakgrad{(g+\frac{\eps}\alpha f)}^2-\weakgrad{ g}^2}{2\frac\eps\alpha},\qquad\alpha\neq 0,
\]
and recalling the linearity in $f$ \eqref{eq:linf}, we obtain 1-homogeneity in $g$, i.e.
\begin{equation}
\label{eq:homg}
\nabla f\cdot\nabla (\alpha g)=\alpha\nabla f\cdot\nabla g,\qquad\forall\alpha\in\R.
\end{equation}
From the locality property \eqref{eq:local} we then get that for $\varphi:\R\to\R$ piecewise affine it holds
\begin{equation}
\label{eq:quasichain}
\nabla f\cdot\nabla(\varphi\circ g)=\varphi'\circ g\nabla f\cdot\nabla g,\qquad\mm\text{\rm-a.e.},
\end{equation}
where we are defining the right hand side as 0 at those $x$ such that $\varphi$ is not differentiable in $g(x)$. To conclude we argue as in the proof of \eqref{eq:chainf} using \eqref{eq:contg} in place of \eqref{eq:1lip}. More precisely, given $\varphi:\R\to\R$ Lipschitz we find a sequence $(\varphi_n)$ of uniformly Lipschitz piecewise affine functions such that $\varphi'_n(z)\to\varphi'(z)$ for $g_\sharp\tilde\mm$-a.e. $z$ ($\tilde\mm$ being as before a probability measure on $X$ having the same negligible sets as $\mm$).

From $\weakgrad{(\varphi\circ g-\varphi_n\circ g)}=|\varphi'-\varphi'_n|\circ g\weakgrad g\to 0$ $\mm$-a.e. and the fact that $\varphi,\varphi_n$, $n\in\N$, are uniformly Lipschitz we get $\lim_{n\to\infty}\int\weakgrad{(\varphi\circ g-\varphi_n\circ g)}^2\,\d\mm\to 0$. Thus from \eqref{eq:contg} and \eqref{eq:quasichain} we conclude
\[
\begin{split}
\int\nabla f\cdot\nabla (\varphi\circ g)&=\lim_{n\to\infty}\int\nabla f\cdot\nabla(\varphi_n\circ g)\,\d\mm\\
&=\lim_{n\to\infty}\int\varphi_n'\circ g\nabla f\cdot\nabla g\,\d\mm\\
&=\int\varphi'\circ g\nabla f\cdot\nabla g\,\d\mm,
\end{split}
\]
having used the dominate convergence theorem in the last step.

The last ingredient we need to prove the symmetry property \eqref{eq:simm} is its integrated version
\begin{equation}
\label{eq:simmint}
\int\nabla f\cdot\nabla g\,\d\mm=\int\nabla g\cdot\nabla f\,\d\mm.
\end{equation}
This easily follows by noticing that the assumption of infinitesimal Hilbertianity yields
\begin{equation}
\label{eq:camel}
\int\frac{\weakgrad{(g+\eps f)}^2-\weakgrad{g}^2}{\eps}-\eps\weakgrad f^2\,\d\mm=\int\frac{\weakgrad{(f+\eps g)}^2-\weakgrad{f}^2}{\eps}-\eps\weakgrad g^2\,\d\mm,
\end{equation}
and letting $\eps\downarrow 0$.

Now notice that \eqref{eq:simm} is equivalent to the fact that for any $h\in L^\infty(X,\mm)$ it holds 
\begin{equation}
\label{eq:trenino}
\int h\nabla f\cdot\nabla g\,\d\mm=\int h\nabla g\cdot\nabla f\,\d\mm,\qquad\forall f,g\in\s^2(X,\sfd,\mm).
\end{equation}
Taking into account the weak$^*$-density of Lipschitz and bounded functions in $L^\infty(X,\mm)$, we easily see that it is sufficient to check \eqref{eq:trenino} for any $h$ Lipschitz and bounded. Also, with the same arguments that led from \eqref{eq:camel} to \eqref{eq:simmint} and a simple truncation argument,  \eqref{eq:trenino} will follow if we show that 
\begin{equation}
\label{eq:simm2}
\s^2(X,\sfd,\mm)\cap L^\infty(X,\mm)\ni f\qquad\mapsto\qquad \int h\weakgrad f^2\,\d\mm\qquad\textrm{is a quadratic form}.
\end{equation}
To this aim, notice that from \eqref{eq:leibf}, \eqref{eq:chaingint} and \eqref{eq:simmint} we get
\begin{equation}
\label{eq:leibsimm}
\begin{split}
\int h\weakgrad f^2\,\d\mm&=\int \nabla (fh)\cdot\nabla f-f\nabla h\cdot\nabla f\,\d\mm\\
&=\int \nabla (fh)\cdot\nabla f-\nabla h\cdot\nabla \big(\frac{f^2}2\big)\,\d\mm=\int \nabla (fh)\cdot\nabla f-\nabla \big(\frac{f^2}2\big)\cdot \nabla h\,\d\mm
\end{split}
\end{equation}
By \eqref{eq:linf} and \eqref{eq:simmint} we know that both $f\mapsto \int \nabla (fh)\cdot\nabla\tilde f\,\d\mm$ and $f\mapsto\int\nabla (\tilde fh)\cdot\nabla f\,\d\mm$ are linear maps, hence $f\mapsto\int \nabla (fh)\cdot\nabla f\,\d\mm$ is a quadratic form. Again by \eqref{eq:linf} we also get that $f\mapsto \int \nabla(\frac {f^2}2)\cdot \nabla h\,\d\mm$ is a quadratic form. Hence \eqref{eq:leibsimm} yields \eqref{eq:simm2} and the conclusion.
\end{proof}
Notice that as a direct byproduct of the proof we get that $\nabla f\cdot \nabla g$ can be realized as limit rather than as infimum, i.e. it holds
\begin{equation}
\label{eq:limite}
\nabla f\cdot\nabla g=\lim_{\eps\to 0}\frac{\weakgrad{(g+\eps f)}^2-\weakgrad g^2}{2\eps},\qquad\forall f,g\in\s^2(\Omega),
\end{equation}
the limit being intended both in $L^2(\Omega)$ and in the pointwise $\mm$-a.e. sense.
\subsection{Definition and basic properties of the Laplacian}
In this section we shall assume that $(X,\sfd,\mm)$ is such that
\begin{equation}
\label{eq:mmslap}
\textrm{$(X,\sfd,\mm)$ is infinitesimally Hilbertian and $(X,\sfd)$ is proper}.
\end{equation}
Given $\Omega\subseteq X$ open, we will denote by  $\test\Omega$ the set of all Lipschitz functions compactly supported in $\Omega$.
\begin{definition}[Laplacian] Let $(X,\sfd,\mm)$ as in \eqref{eq:mmslap} and $\Omega\subseteq X$ open.
Let $g:\Omega\to\R$ be a locally Lipschitz function. We say that $g$ has a distributional Laplacian in $\Omega$, and write $g\in D(\bd,\Omega)$, provided there exists a locally finite Borel measure $\mu$ on $\Omega$ such that
\begin{equation}
\label{eq:deflap}
-\int\nabla f\cdot\nabla g\,\d\mm=\int f\,\d\mu,\qquad\forall f\in\test\Omega.
\end{equation}
In this case we will say that $\mu$ (which is clearly unique) is the distributional Laplacian of $g$ and  indicate it by $\bd g\restr\Omega$. 
\end{definition}
Notice that the integrand in the left hand side of \eqref{eq:deflap} is in $L^1(\Omega)$, because $g$, being locally Lipschitz, is Lipschitz on $\supp(f)$. Also, since the left hand side of \eqref{eq:deflap} is linear in $g$, the set $D(\bd,\Omega)$ is a vector space and the map
\[
D(\bd,\Omega)\ni g\qquad\mapsto\qquad\bd g\restr\Omega\,,
\]
is linear. We now show that the basic calculus rules for the Laplacian are true also in this generality.
\begin{proposition}[Chain rule]\label{prop:chainlap}
Let $(X,\sfd,\mm)$ be as in \eqref{eq:mmslap},  $\Omega\subseteq X$  an open set and $g:\Omega\to\R$ a locally Lipschitz function in $D(\bd,\Omega)$. Then for every  function $\varphi\in C^{1,1}_{loc}(g(\Omega))$, the function $\varphi\circ g$ is in $D(\bd,\Omega)$ and it holds  
\begin{equation}
\label{eq:chainlap}
\bd(\varphi\circ g)\restr\Omega=\varphi'\circ g\bd g\restr\Omega+\varphi''\circ g\weakgrad g^2\mm\restr\Omega.
\end{equation}
\end{proposition}
\begin{proof} The right hand side of \eqref{eq:chainlap} defines a locally finite measure, so the statement makes sense.

Let $f\in\test\Omega$ and notice that being $\varphi'\circ g$ locally Lipschitz, we also have  $f\varphi'\circ g\in\test\Omega$. With a simple locality argument  one can check that the chain rule \eqref{eq:chainf} and the symmetry \eqref{eq:simm} yield \[
\begin{split}
\nabla f\cdot\nabla(\varphi\circ g)&=\varphi'\circ g\nabla f\cdot\nabla g,\\
\nabla(\varphi'\circ g)\nabla g&=\varphi''\circ g\nabla g\cdot\nabla g,
\end{split}
\] 
$\mm$-a.e. on $\Omega$. Hence, taking into account  the Leibniz rule \eqref{eq:leibf} and the identity \eqref{eq:boundfg}, we have
\[
\begin{split}
-\int \nabla f\cdot\nabla(\varphi\circ g)\,\d\mm&=-\int \varphi'\circ g\nabla f\cdot\nabla g\,\d\mm\\
&=-\int\nabla(f\varphi'\circ g)\cdot\nabla g-f\nabla(\varphi'\circ g)\nabla g\,\d\mm\\
&=\int f\varphi'\circ g\,\d\bd g\restr\Omega+\int f\varphi''\circ g\weakgrad g^2\,\d\mm,
\end{split}
\]
which is the thesis.
\end{proof}
\begin{proposition}[Leibniz rule]
Let $(X,\sfd,\mm)$ be as in \eqref{eq:mmslap},  $\Omega\subseteq X$  an open set  and $g_1,g_2\in D(\bd,\Omega)$. Then $g_1g_2\in D(\bd,\Omega)$ and
\begin{equation}
\label{eq:leiblap}
\bd(g_1g_2)\restr\Omega=g_1\bd g_2\restr\Omega+g_2\bd g_1\restr\Omega+2\nabla g_1\cdot\nabla g_2\mm\restr\Omega.
\end{equation}
\end{proposition}
\begin{proof}
The right hand side of \eqref{eq:leiblap} defines a locally finite measure, so the statement makes sense. For $f\in\test\Omega$ we have $fg_1,fg_2\in\test\Omega$, hence using the Leibniz rule \eqref{eq:leibf} and the symmetry \eqref{eq:simm}  we get
\[
\begin{split}
-\int\nabla f\cdot\nabla(g_1g_2)\,\d\mm&=-\int g_1\nabla f\cdot\nabla g_2+g_2\nabla f\cdot\nabla g_1\,\d\mm\\
&=-\int \nabla(fg_1)\cdot\nabla g_2+ \nabla(fg_2)\cdot\nabla g_1-2f\nabla g_1\cdot\nabla g_2\,\d\mm\\
&=\int fg_1\,\d\bd g_2\restr\Omega+\int fg_2\,\d\bd g_1\restr\Omega+\int2f\nabla g_1\cdot\nabla g_2\,\d\mm,
\end{split}
\]
which is the thesis.
\end{proof}
\begin{proposition}[Comparison]\label{prop:compar}
Let $(X,\sfd,\mm)$ be as in \eqref{eq:mmslap},  $\Omega\subseteq X$  an open set, $g:\Omega\to\R$ locally Lipschitz and assume that there exists a locally finite measure $\mu$ on $\Omega$ such that
\[
-\int\nabla f\cdot\nabla g\,\d\mm\leq \int f\,\d\mu,\qquad\forall f\in\test\Omega,\ f\geq 0.
\]
Then $g\in D(\bd,\Omega)$ and $\bd g\restr\Omega\leq \mu$.
\end{proposition}
\begin{proof}
The map
\[
\test\Omega\ni f\qquad\mapsto L(f):= \int f\,\d\mu+\int\nabla f\cdot\nabla g\,\d\mm,
\]
is linear and satisfies $L(f)\geq 0$ for $f\geq 0$. To get the proof, it is sufficient to show that there exists a non-negative Radon measure $\tilde\mu$ on $\Omega$ such that $L(f)=\int f\,\d\tilde\mu$ for any $f\in \test\Omega$. 

To this aim, fix a compact set $K\subset\Omega$ and a  function $\nchi_K\in \test\Omega$ such that $0\leq\nchi_K\leq 1$ everywhere and $\nchi_K= 1$ on $K$. Let $V_K\subset \test\Omega$ be the set of Lipschitz  functions with support contained in $K$ and observe that for any  $f\in V_K$, the fact that $(\max |f|)\nchi_K-f$ is in $\test\Omega$ and non-negative yields
\[
\begin{split}
0\leq L\big((\max|f|)\nchi -f\big)=(\max|f|)L(\nchi)-L(f).
\end{split}
\]
Replacing $f$ with $-f$ we deduce
\[
|L(f)|\leq (\max|f|) L(\nchi),
\]
i.e. $L:V_K\to\R$ is continuous w.r.t. the $\sup$ norm. Hence it can be extended to a (unique, by the density of Lipschitz functions in the uniform norm) linear bounded functional on the set $C_K\subset C(X)$ of continuous functions with support contained in $K$. Since $K$ was arbitrary, by the Riesz theorem we get that there exists a locally finite Borel measure $\tilde\mu$ such that
\[
L(f)=\int f\,\d\tilde\mu,\qquad\forall f\in \test\Omega.
\]
Clearly $\tilde\mu$ is non-negative, thus the thesis is achieved.
\end{proof}
\subsection{Proof of Laplacian comparison}
In this section we present a simplified proof of the Laplacian comparison estimates for the distance on infinitesimally Hilbertian $CD(K,N)$ spaces.

Before going into technicalities we outline the strategy in the Riemannian setting, the point being that we can only rely on the $CD(K,N)$ condition and not on Jacobi fields calculus nor on the Bochner inequality, which as of today are not available in the non-smooth setting. Consider for simplicity $K=0$, so that our data is a Riemannian manifold $M$ such that for any couple of probability measures $\mu_0,\mu_1$ it holds
\[
\u_N(\mu_t)\leq (1-t)\u_N(\mu_0)+t\u_N(\mu_1),
\] 
where $(\mu_t)$ is any $W_2$-geodesic connecting $\mu_0$ to $\mu_1$, and the reference measure used to compute $\u_N$ is the volume measure ${\rm vol}$. Let $\overline x\in M$ be a point, $\rho$ a smooth probability density and apply the previous inequality with $\mu_0:=\rho{\rm vol}$ and $\mu_1=\delta_{\overline x}$ to get
\begin{equation}
\label{eq:11}
\lim_{t\to 0}\frac{\u_N(\mu_t)-\u_N(\mu_0)}{t}\leq -\u_N(\mu_0)=\int\rho^{1-\frac1N}\,\d{\rm vol},
\end{equation}
having used the fact that $\u_N(\delta_{\overline x})=0$. On the other hand, letting $\sfd$ be the Riemannian distance and $g(x):=\frac{1}2\sfd^2(x,\overline x)$, a direct explicit computation based on the fact that $\mu_t=\exp(-t\nabla g)_\sharp\mu_0$  gives
\begin{equation}
\label{eq:21}
\lim_{t\to 0}\frac{\u_N(\mu_t)-\u_N(\mu_0)}{t}=-\frac1N\int\nabla \rho^{1-\frac1N}\cdot\nabla g\,\d{\rm vol}.
\end{equation}
Combining \eqref{eq:11} and \eqref{eq:21} yields, thanks to arbitrariness of $\rho$, the sharp Laplacian comparison $\Delta g\leq N{\rm vol}$.

\bigskip

We will follow this argument in proving the Laplacian comparison for the non-smooth case. Notice that while inequality \eqref{eq:11} can  certainly be proved in the same way in the abstract situation, for what concerns the derivative computed in \eqref{eq:21} things are more difficult due to the lack of a change of variable formula. 

The crucial technical fact that we will use is the following lemma. Its statement is  a non-smooth analogous of the fact that if we differentiate a smooth function $f$ along a geodesic $(\gamma_t)$ connecting  a point $x\in M$ to the reference point $\overline x\in M$ we obtain $\nabla f\cdot \nabla g$, with $g:=\frac12\sfd^2(\cdot,\overline x)$ as above (say that we are far from the cut-locus). In this sense it relates the `horizontal' derivative $\lim_{t\to 0}\frac{f(\gamma_t)-f(x)}{t}$ obtained by perturbing $f$ in the independent variable with the `vertical' derivative $\lim_{\eps\to 0}\frac{|\nabla(g+\eps f)|^2(x)-|\nabla g|^2(x)}{2\eps}$ obtained by perturbing $g$ in the dependent variable that we used in Section \ref{se:infhil} to define the object $\nabla f\cdot \nabla g$.

We also remark that the expression \eqref{eq:horver} can be seen as an infinitesimal version of the change of variable formula.
\begin{lemma}[Horizontal and vertical derivatives]\label{le:horver}
Let $(X,\sfd,\mm)$ be an infinitesimally Hilbertian $CD(K,N)$ space, $K\in\R$, $N\in(1,\infty)$, and $\Omega\subseteq X$ an open set. Let $\bar x\in \Omega$, $\mu=\rho\mm\in\probt X$ a measure with bounded density and bounded support concentrated on $\Omega$ and let  $\ppi\in\gopt(\mu,\delta_{\bar x})$ be a plan such that \eqref{eq:cd} holds and $(\e_t)_\sharp\ppi$ is concentrated on $\Omega$ for every $t\in[0,1]$.

Then for every Lipschitz function $f:\Omega\to\R$ it holds
\begin{equation}
\label{eq:horver}
\lim_{t\downarrow0}\int \frac{f\circ\e_t-f\circ\e_0}{t}\,\d\ppi=-\int\nabla f\cdot\nabla g\,\rho\,\d\mm,
\end{equation}
where $g:=\frac{\sfd^2(\cdot,\bar x)}{2}$.
\end{lemma}
\begin{proof}
Let $h:\Omega\to\R$ be a Lipschitz function and recall that the local Lipschitz constant $|Dh|$ is an upper gradient of $h$ so that
\[
h(\gamma_t)-h(\gamma_0)\leq\int_0^t|Dh|(\gamma_t)|\dot\gamma_t|\,\d t\leq \frac12\int_0^t|Dh|^2(\gamma_t)\,\d t+\frac12\int_0^t|\dot\gamma_t|^2\,\d t,
\]
for any absolutely continuous curve $\gamma$. Divide by $t$ and integrate this inequality w.r.t. $\ppi$ to get
\begin{equation}
\label{eq:perhorver}
\int\frac{h(\gamma_t)-h(\gamma_0)}{t}\,\d\ppi(\gamma)\leq \frac1{2t}\iint_0^t|Dh|^2(\gamma_t)\,\d t\,\d\ppi(\gamma)+\frac1{2t}\iint_0^t|\dot\gamma_t|^2\,\d t\,\d\ppi(\gamma).
\end{equation}
We claim that
\begin{equation}
\label{eq:claimbb}
\lim_{t\to 0}\frac1t\iint_0^t|Dh|^2(\gamma_t)\,\d t\,\d\ppi(\gamma)=\int|Dh|^2\rho\,\d\mm.
\end{equation}
Indeed, let $\mu_t:=(\e_t)_\sharp\ppi$ and $\nu_t:=\frac1t\int_0^t\mu_t\,\d t$ and notice that  $(\mu_t)$ weakly converges to $\mu$ as $t\downarrow 0$ in duality with $C_b(X)$. Letting $t\downarrow 0$ in \eqref{eq:cd} and using the lower semicontinuity of $\u_N$ w.r.t. weak convergence in duality with $C_b(X)$ we get $\lim_{t\downarrow0}\u_N(\mu_t)=\u_N(\mu)$. Using the convexity (w.r.t. standard affine interpolation) and the weak semicontinuity of $\u_N$ again we also obtain  $\lim_{t\downarrow0}\u_N(\nu_t)=\u_N(\mu)$. Our aim is to prove that $\lim_{t\downarrow 0}\int|Dh|^2\,\d\nu_t=\int |Dh|^2\rho\,\d\mm$. It is readily checked that $(\nu_t)$ converges to $\mu$ as $t\downarrow0$ in duality with $C_b(X)$, so that if $|Dh|^2$ is continuous there is nothing to prove. In general however, $|Dh|^2$ is just a bounded and Borel function, so that in order to get the desired convergence we need to prove that $(\nu_t)$ converges to $\mu$ in duality with Borel and bounded functions. In other words, we have to prove that there exists $c>0$ such that 
\[
\lim_{t\downarrow0}\int_{\{\rho_t>c\}}\rho_t\,\d\mm+\nu_t^s(X)=0,
\]
where $\nu_t=\rho_t\mm+\nu^s_t$, $\nu_t^s\perp\mm$. We will get this by using the fact that $\lim_{t\downarrow0}\u_N(\nu_t)=\u_N(\mu)$. Argue by contradiction and assume that there exists sequences $t_n\downarrow 0$ and $c_n\uparrow\infty$ such that 
\begin{equation}
\label{eq:lossmass}
\int_{\{\rho_{t_n}\geq c_n\}}\rho_{t_n}\,\d\mm+\nu^s_{t_n}(X)\geq C>0,\qquad\forall n\in\N.
\end{equation}
Notice that the measures $\{\nu_{t_n}\}_{n\in\N}$ are all concentrated on a single bounded set and thus, being $(X,\sfd)$ is proper, they are all concentrated on the same compact set $K$.  Define the measures $\tilde\nu_{t_n}:=\rho_{t_n}\nchi_{\{\rho_{t_n}\leq c_n\}}\mm\leq \nu_{t_n}$ and notice that $|\tilde \nu_{t_n}|(X)\leq 1-C$ and that all the $\tilde\nu_{t_n}$'s are concentrated on $K$. It easily follow that up to subsequences - not relabeled - the sequence  $(\tilde\nu_{t_n})$ weakly converges to some Borel measure $\tilde\nu$ in duality with $C_b(X)$. From
\[
\begin{split}
\big|\u_N(\nu_{t_n})-\u_N(\tilde\nu_{t_n})\big|=\int_{\{\rho_{t_n}\geq c_n\}}|u_N(\rho_{t_n})|\,\d\mm\leq \left(\int_{\{\rho_{t_n}\geq c_n\}}\rho_{t_n}\,\d\mm\right)^{\frac{N-1}{N}}\mm(\{\rho_{t_n}\geq c_n\})^{\frac1N},
\end{split}
\]
and Chebyshev's inequality $\mm(\{\rho_{t_n}\geq c_n\})\leq\frac{1}{c_n}\downarrow 0$, we get $\lim_{n\to\infty}|\u_N(\nu_{t_n})-\u_N(\tilde\nu_{t_n})|=0$ and thus
\begin{equation}
\label{eq:ass}
\u_N(\tilde\nu)\leq\limi_{n\to\infty}\u_N(\tilde\nu_{t_n})=\limi_{n\to\infty}\u_N(\nu_{t_n})=\u_N(\mu).
\end{equation}
Write $\tilde\nu=\eta\mm+\nu^s$ with $\nu^s\perp\mm$ and notice that by construction it holds $\tilde\nu\leq\mu$, so that $\nu^s=0$ and $\eta\leq\rho$ $\mm$-a.e.. The fact that  $u_N$ is strictly decreasing and \eqref{eq:ass} then forces $\eta=\rho$ $\mm$-a.e., so that  $\tilde\nu(X)=\int\eta\,\d\mm=\int\rho\,\d\mm=1$. But this is a contradiction, because from \eqref{eq:lossmass} we have $\tilde\nu_{t_n}(X)\leq 1-C$ for every $n\in\N$, so that $\tilde\nu(X)\leq1-C$ as well.

Thus \eqref{eq:claimbb} holds. Therefore passing to the limit in \eqref{eq:perhorver}, using \eqref{eq:cheeger2} and the fact that $|\dot\gamma_t|\equiv\sfd(\gamma_0,\gamma_1)=\sfd(\gamma_0,\bar x)$ for $\ppi$-a.e. $\gamma$, we get
\begin{equation}
\label{eq:perhorver2}
\lims_{t\downarrow0}\int\frac{h(\gamma_t)-h(\gamma_0)}{t}\,\d\ppi(\gamma)\leq \frac1{2}\int\weakgrad h^2\rho\,\d\mm+\frac1{2}\int\sfd^2(\cdot,\bar x)\rho\,\d\mm.
\end{equation}
Since  $\ppi$-a.e. $\gamma$ is a geodesic with $\bar x$ as endpoint, for $\ppi$-a.e. $\gamma$ it holds $\lim_{t\downarrow0}\frac{g(\gamma_t)-g(\gamma_0)}{t}=-|Dg|(\gamma_0)=-\sfd(\gamma_0,\bar x)$, therefore using  \eqref{eq:cheeger2}   again we obtain
\begin{equation}
\label{eq:perhorver3}
\lim_{t\downarrow0}\int\frac{g(\gamma_t)-g(\gamma_0)}{t}\,\d\ppi(\gamma)=-\frac1{2}\int\weakgrad{(- g)}^2\rho\,\d\mm-\frac1{2}\int\sfd^2(\cdot,\bar x)\rho\,\d\mm.
\end{equation}
Write \eqref{eq:perhorver2} for $h:=-g+\eps f$ and add  \eqref{eq:perhorver3} to get
\[
\lims_{t\downarrow0}\eps\int\frac{f(\gamma_t)-f(\gamma_0)}{t}\,\d\ppi(\gamma)\leq \frac1{2}\int\Big(\weakgrad{(-g+\eps f)}^2-\weakgrad{(-g)}^2\Big)\rho\,\d\mm.
\]
Divide by $\eps>0$ and let $\eps\downarrow0$ to get, recalling \eqref{eq:limite}, the inequality
\[
\lims_{t\downarrow0}\int\frac{f(\gamma_t)-f(\gamma_0)}{t}\,\d\ppi(\gamma)\leq -\int\nabla f\cdot\nabla g\,\rho\,\d\mm.
\]
Replace $f$ with $-f$ and use the linearity in $f$ of the right hand side to conclude.
\end{proof}

\begin{proposition}[Lower bound on the derivative of $\u_N$]\label{prop:lowerbound}
Let $(X,\sfd,\mm)$ be an infinitesimally Hilbertian $CD(K,N)$ space, $K\in \R$, $N\in(1,\infty)$, $\bar x\in X$ and $R>0$. Also, let  $\mu=\rho\mm$ be a measure concentrated on $B_R(\bar x)$ such that $\rho:B_R(\bar x)\to[0,\infty)$ is Lipschitz and bounded from below by some positive constant. 

Then for every $\ppi\in\gopt(\mu,\delta_{\bar x})$ such that \eqref{eq:cd} holds we have
\begin{equation}
\label{eq:upperbound}
\limi_{t\downarrow0}\frac{\u_N((\e_t)_\sharp\ppi)-\u_N(\mu)}t\geq-\int_{B_R(\bar x)} \nabla(\rp_N\circ \rho)\cdot \nabla g\,\d\mm,
\end{equation}
where  $g:=\frac{\sfd^2(\cdot,\bar x)}{2}$ and the pressure functional $\rp_N:\R^+\to\R$ is defined by $\rp_N(z):=z u_N'(z)-u_N(z)$.
\end{proposition}
\begin{proof}
Start noticing that since $\rho,\rho^{-1}: B_R(\bar x)\to\R$  are Lipschitz and bounded, the function $u_N'(\rho):{B_R}(\bar x)\to\R$ is Lipschitz and bounded as well.  Thus for every $\nu\in\probt{X}$ absolutely continuous w.r.t. $\mm\restr{B_R(\bar x)}$, the convexity of $u_N$ gives $\u_N(\nu)-\u_N(\mu)\geq \int_{B_R(\bar x)} u_N'(\rho)(\frac{\d\nu}{\d\mm}-\rho)\,\d\mm$. Then a simple approximation argument based on the continuity of $\rho$ gives
\[
\u_N(\nu)-\u_N(\mu)\geq \int_{B_R(\bar x)} u_N'(\rho)\,\d\nu-\int_{B_R(\bar x)} u_N'(\rho)\,\d\mu,
\]
for any $\nu\in\probt{X}$ concentrated on $B_R(\bar x)$. Plugging $\nu:=(\e_t)_\sharp\ppi$, dividing by $t$ and letting $t\downarrow0$ we get
\[
\limi_{t\downarrow 0}\frac{\u_N((\e_t)_\sharp\ppi)-\u_N((\e_0)_\sharp\ppi)}t\geq\limi_{t\downarrow0}\int\frac{u_N'(\rho)\circ\e_t-u_N'(\rho)\circ\e_0}t\,\d\ppi.
\]
Now observe that all the assumptions of Lemma \ref{le:horver} are fulfilled with $\Omega:=B_R(\bar x)$ and $f:=u_N'\circ \rho$. Hence \eqref{eq:horver} yields
\[
\limi_{t\downarrow 0}\frac{\u_N((\e_t)_\sharp\ppi)-\u_N((\e_0)_\sharp\ppi)}t\geq-\int\nabla( u_N'\circ\rho)\cdot\nabla g\,\rho\,\d\mm.
\]
To conclude, notice that $\rp_N'(z)=zu_N''(z)$ and apply twice the chain rule \eqref{eq:chainf}:
\[
 \int_{B_R(\bar x)} \nabla (u_N'\circ \rho)\cdot \nabla g\, \rho\,\d\mm=\int_{B_R(\bar x)} \rho \, u_N''\circ \rho\nabla \rho\cdot \nabla g\,\d\mm=\int_{B_R(\bar x)} \nabla (\rp_N\circ \rho)\cdot \nabla g \,\d\mm.
\]

\end{proof}

For $K\in \R$, $N\in(1,\infty)$, we introduce the functions $\tilde\tau_{K,N}:[0,\infty)\to \R$ by
\[
\tilde\tau_{K,N}(\theta):=
\left\{
\begin{array}{ll}
\dfrac 1N\left(1+\theta\sqrt{K(N-1)}\,{\rm cotan}\left(\theta\sqrt{\frac{K}{N-1}}\right) \right),&\qquad\textrm{if }K>0,\\
\\
1,&\qquad\textrm{if }K=0,\\
\\
\dfrac 1N\left(1+\theta\sqrt{-K(N-1)}\,{\rm cotanh}\left(\theta\sqrt{\frac{-K}{N-1}}\right) \right),&\qquad\textrm{if }K<0.
\end{array}
\right.
\]
Notice that $\tilde\tau_{K,N}$ satisfies
\begin{equation}
\label{eq:tautilde}
\tilde\tau_{K,N}(\theta)=\lim_{t\downarrow0}\frac{\tau^{(1)}_{K,N}(\theta)-\tau^{(1-t)}_{K,N}(\theta)}{t},
\end{equation}
(provided $K\theta^2<(N-1)\pi^2$ if $K>0$). 
\begin{proposition}[Upper bound on the derivative of $\u_N$]\label{prop:upperbound}
Let $(X,\sfd,\mm)$ be an infinitesimally Hilbertian $CD(K,N)$ space, $K\in\R$, $N\in(1,\infty)$, $\mu=\rho\mm$ a measure with bounded support and bounded density, $\bar x\in X$ and $\ppi\in\gopt(\mu,\delta_{\bar x})$ such that \eqref{eq:cd} is satisfied. 

Then 
\begin{equation}
\label{eq:upper}
\lims_{t\downarrow 0}\frac{\u_N((\e_t)_\sharp\ppi)-\u_N((\e_0)_\sharp\ppi)}{t}\leq \int_X\tilde\tau_{K,N}\big(\sfd(\cdot,\bar x)\big)\rho^{1-\frac1N}\,\d\mm.
\end{equation}
\end{proposition}
\begin{proof}
From $\u_N((\e_0)_\sharp\ppi)=-\int\rho^{-\frac1N}(\gamma_0)\,\d\ppi(\gamma)$, and \eqref{eq:cd}  we get
\[
\frac{\u_N((\e_t)_\sharp\ppi)-\u_N((\e_0)_\sharp\ppi)}{t}\leq\int\rho^{1-\frac1N}\bigg(\frac{1-\tau^{(1-t)}_{K,N}\big(\sfd(\cdot,\bar x)\big)}t\bigg)\,\d\mm.
\]
If $K\leq 0$ the incremental ratios in the right hand side of \eqref{eq:tautilde} are uniformly bounded, thus the dominate convergence theorem gives the conclusion. If $K>0$ and $\sup_{x\in\supp(\mu)}\sfd^2(x,\bar x)<\frac{(N-1)\pi^2}{K}$, then again  the incremental ratios in   \eqref{eq:tautilde} are uniformly bounded  we conclude as before.

Therefore, assume that $K>0$ and $\sup_{x\in\supp(\mu)}\sfd^2(x,\bar x)=\frac{(N-1)\pi^2}{K}=:\sf D^2$.  Since the convergence of the incremental ratios at the right hand side of \eqref{eq:tautilde} is monotone increasing, taking into account that $\mu\leq c\mm$ for some $c>0$, to conclude it is sufficient to show that $\tilde\tau_{K,N}(\sfd(\cdot,\bar x))\in L^1(X,\mm)$.

By compactness there exists $x_0\in\supp(\mu)$ such that $\sfd(x_0,\bar x)=\sf D $. We claim that  it holds
\begin{equation}
\label{eq:palle}
\mm(B_{r}(x_0))=\frac{\int_0^r\sin(t\pi/{\sf D})^{N-1}\,\d t}{\int_0^{\sf D}\sin(t\pi/{\sf D})^{N-1}\,\d t}\mm(X),\qquad\forall r\in[0,{\sf D}].
\end{equation}
Indeed, the two balls $B_{r}(x_0)$ and $B_{{\sf D}-r}(\bar x)$ are disjoint, so that by  the Bishop-Gromov volume comparison estimates (inequality \eqref{eq:BGv}) we have
\begin{equation}
\label{eq:sonno}
\begin{split}
\mm(X)&\geq \mm(B_{r}(\bar x))+\mm(B_{{\sf D}-r}(x_0))\\
&\geq \left( \frac{\int_0^r\sin(t\pi/{\sf D})^{N-1}\,\d t}{\int_0^{\sf D}\sin(t\pi/{\sf D})^{N-1}\,\d t}+ \frac{\int_0^{{\sf D}-r}\sin(t\pi/{\sf D})^{N-1}\,\d t}{\int_0^{\sf D}\sin(t\pi/{\sf D})^{N-1}\,\d t}\right)\!\mm(X)= \mm(X).
\end{split}
\end{equation}
Hence the inequalities for each term must be equalities and our claim is true. It follows also that 
\begin{equation}
\label{eq:diameter}
\sfd(x,\bar x)+\sfd(x,x_0)={\sf D},\qquad\forall x\in\supp(\mm),
\end{equation}
or otherwise for some $r\in[0,{\sf D}]$ the union of the balls $\mm(B_{r}(\bar x))$ and $\mm(B_{{\sf D}-r}(x_0))$ would not cover some neighborhood of an $x\in\supp(\mm)$ thus violating the fact that inequalities in \eqref{eq:sonno} are equalities.

Let $T:X\to[0,{\sf D}]$ be given by $T(x):=\sfd(x,\bar x)$. The identity \eqref{eq:palle} gives
\[
\d T_\sharp\mm(r)=\left(\frac{\sin(r\pi/{\sf D})^{N-1}}{\int_0^{\sf D}\sin(t\pi/{\sf D})^{N-1}\,\d t}\mm(X)\right)\d\mathcal L^1\restr{[0,{\sf D}]}(r).
\]
To conclude, notice that by the very definition of $\tilde\tau_{K,N}$, to prove that  $\tilde\tau_{K,N}(\sfd(\cdot,\bar x))\in L^1(X,\mm)$ it suffices to prove that 
\[
\int\limits_{B_{{\sf D}/2}(x_0)} \frac{1}{\sin\big(\sfd(x,\bar x)\sqrt{\frac K{N-1}}\big)}\,\d\mm(x)<\infty.
\]
This follows from 
\[
\begin{split}
\int\limits_{B_{{\sf D}/2}(x_0)} \frac{1}{\sin\big(\sfd(x,\bar x)\sqrt{\frac K{N-1}}\big)}\,\d\mm(x)&=\int\limits_{B_{{\sf D}/2}(x_0)} \frac{1}{\sin\big(\sfd(x,x_0)\frac\pi{\sf D}\big)}\,\d\mm(x)\\
&=\frac{\mm(X)}{\int_0^{\sf D}\sin(t\pi/{\sf D})^{N-1}\,\d t}\int_0^{{\sf D}/2}\sin\big(r\frac\pi{\sf D}\big)^{N-2}\,\d r,
\end{split}
\]
and the fact that $N>1$.
\end{proof}
\begin{theorem}[Laplacian comparison]
Let $(X,\sfd,\mm)$ be an infinitesimally Hilbertian $CD(K,N)$ space, $K\in \R$, $N\in(1,\infty)$, and $\bar x\in X$. 

Then $\frac{\sfd^2(\cdot,\bar x)}2\in D(\bd,X)$ and
\begin{equation}
\label{eq:lapsq}
\bd\frac{\sfd^2(\cdot,\bar x)}2\leq N\,\tilde\tau_{K,N}(\sfd(\cdot,\bar x))\,\mm.
\end{equation}
\end{theorem}
\begin{proof}
Let $g:=\frac{\sfd^2(\cdot,\bar x)}{2}$. By Proposition \ref{prop:compar} it is sufficient to show that 
\begin{equation}
\label{eq:percomp}
-\int_X\nabla f\cdot\nabla g\,\d\mm\leq N\int_X f\,\tilde\tau_{K,N}(\sfd(\cdot,\bar x))\,\d\mm,\qquad\forall f\in\test X,\ f\geq 0.
\end{equation}
Thus, fix a non-negative $f\in\test X$ and let $R>0$ be such that $\supp(f)\subset B_R(\bar x)$. For $\eps>0$ define $\rho_\eps:X\to[0,\infty)$ as $\rho_\eps(x):=c_\eps(f+\eps)^{\frac N{N-1}}\nchi_{B_R(x)}$, where $c_\eps$ is such that $\int\rho_\eps\,\d\mm=1$. Let $\mu:=\rho\mm$ and $\ppi\in\gopt(\mu,\delta_{\bar x})$ be such that \eqref{eq:cd} holds. 

Propositions  \ref{prop:lowerbound} and \ref{prop:upperbound} are both applicable with $\rho_\eps$ in place of $\rho$, hence we get
\[
\begin{split}
-\int_{B_R(\bar x)} \nabla(\rp_N(\rho_\eps))\cdot \nabla g\,\d\mm&\leq \limi_{t\downarrow0}\frac{\u_N((\e_t)_\sharp\ppi)-\u_N(\mu)}t\\
&\leq\lims_{t\downarrow0}\frac{\u_N((\e_t)_\sharp\ppi)-\u_N(\mu)}t\leq\int_X\tilde\tau_{K,N}\big(\sfd(\cdot,\bar x)\big)\rho^{1-\frac1N}_\eps\,\d\mm.
\end{split}
\]
Recall that $\rp_N(z):=zu'_N(z)-u_N(z)=\frac1Nz^{1-\frac1N}$ so that by the very definition of $\rho_\eps$ we have
\[
-\frac1N\int_{B_R(\bar x)}\nabla(f+\eps)\cdot\nabla g\,\d\mm\leq \int_{B_R(\bar x)}\tilde\tau_{K,N}\big(\sfd(\cdot,\bar x)\big)(f+\eps)\,\d\mm,
\]
and letting $\eps\downarrow0 $ we get \eqref{eq:percomp} and the conclusion.
\end{proof}
In the proof of the excess estimate we will need the Laplacian comparison estimate for the distance function, rather than for the squared distance. This can be directly obtained from inequality \eqref{eq:lapsq} and the chain rule in Proposition \ref{prop:chainlap}. The corresponding comparison is then:
\begin{equation}
\label{eq:lapdis}
\sfd(\cdot,\bar x)\in D(\bd,X\setminus \{\bar x\}),\qquad\qquad\bd\sfd(\cdot,\bar x)\restr{X\setminus \{\bar x\}}\leq \frac{N\,\tilde\tau_{K,N}(\sfd(\cdot,\bar x))-1}{\sfd(\cdot,\bar x)}\mm.
\end{equation}

\subsection{The weak maximum principle}

Here we show that the notion of Laplacian we introduced allows an easy proof of the weak maximum principle, i.e. that any Lipschitz $g:\overline\Omega\to\R$ with non-negative Laplacian achieves its maximum on $\partial\Omega$. This will be sufficient for our purposes, but we remark that the strong maximum principle actually holds, i.e. that if under the same assumptions $g$ has a maximum inside $\Omega$, then it is constant on the connected component containing the maximum. This stronger result can be achieved by techniques related to non-linear potential theory (see \cite{BjornBjorn11}) and the latters are directly linked to the definition of Laplacian we gave (see \cite{GM12}). Also, we will state the result for $CD(K,N)$ spaces, but actually everything holds under just doubling\&Poincar\'e assumptions.

We start with the following statement.
\begin{proposition}[Relation with super/sub minimizers]\label{prop:relmaxmin}
Let $(X,\sfd,\mm)$ be as in \eqref{eq:mmslap},  $\Omega\subseteq X$ an open set and $g:\Omega\to\R$ a Lipschitz function. Then the following are equivalent.
\begin{itemize}
\item[i)] $g\in D(\bd,\Omega)$ and $\bd g\restr\Omega\geq 0$ (resp. $\bd g\restr\Omega\leq 0$).
\item[ii)] For any non-positive (resp. non-negative) function $f\in\test\Omega$ it holds
\begin{equation}
\label{eq:minimizers}
\int_\Omega\weakgrad g^2\,\d\mm\leq \int_\Omega\weakgrad{(g+f)}^2\,\d\mm.
\end{equation}
\end{itemize} 
\end{proposition}
\begin{proof}$\ $\\
\noindent{$\mathbf{(i)\Rightarrow(ii)}$} Pick $f\in\test\Omega$ and notice that the function $\eps\mapsto\int_\Omega\weakgrad{(g+\eps f)}^2\,\d\mm$ is convex, so that it holds
\[
\int_\Omega\weakgrad{(g+f)}^2\,\d\mm-\int_\Omega\weakgrad{g}^2\,\d\mm\geq \lim_{\eps\downarrow 0}\int_\Omega\frac{\weakgrad{(g+\eps f)}^2-\weakgrad{g}^2}{\eps}\,\d\mm=2\int_\Omega \nabla f\cdot\nabla g\,\d\mm,
\]
having used \eqref{eq:limite} in the last passage. The right hand side of this expression equals,  by definition,  $-2\int_\Omega f\,\d\bd g\restr\Omega$, thus from the assumption $\bd g\restr\Omega\geq 0$ we get
\begin{equation}
\label{eq:divano}
\int_\Omega\weakgrad{(g+f)}^2\,\d\mm-\weakgrad{g}^2\,\d\mm\geq0,\qquad\forall f\in\test\Omega,\ f\leq 0,
\end{equation}
as desired. The case $\bd g\restr\Omega\leq 0$ is handled analogously.

\noindent{$\mathbf{(ii)\Rightarrow(i)}$} Pick a non-positive $f\in\test\Omega$ and recall that, as before, we have
\[
\int_\Omega \nabla f\cdot\nabla g\,\d\mm=\lim_{\eps\downarrow 0}\int_\Omega\frac{\weakgrad{(g+\eps f)}^2-\weakgrad{g}^2}{2\eps}\,\d\mm.
\]
By assumption the right hand side of the above expression is $\geq 0$. Thus, the conclusion comes from Proposition \ref{prop:compar}.  The other case is analogous.
\end{proof}
The weak maximum principle for  subminimizers can be easily obtained through a truncation argument.
\begin{proposition}[Weak maximum principle]\label{thm:maxprinc}
Let $(X,\sfd,\mm)$ be a $CD(K,N)$ space,  $\Omega\subseteq X$ an open set and $g:\overline{\Omega}\to\R$ a Lipschitz function such that 
\[
\int_\Omega\weakgrad g^2\,\d\mm\leq \int_\Omega\weakgrad{(g+f)}^2\,\d\mm,\qquad\forall f\in\test\Omega,\ f\leq 0.
\]
Then 
\beq
\label{eq:maxboundary}
\max_{\overline{\Omega}} g\leq \max_{\partial\Omega} g.
\eeq
\end{proposition}

\begin{proof}
Let $M:=\max_{\overline\Omega} g$ and suppose \eqref{eq:maxboundary} is false. Therefore, there exists $x_0\in\Omega$ and $\eps>0$ such that
\begin{equation}
\label{eq:perweak}
g(x_0)>M-\eps>\max_{\partial\Omega} g+\eps.
\end{equation}
Put $f:=-(g-M+\eps)^+\in \test\Omega$, and notice that, using  \eqref{eq:localweak}, \eqref{eq:nullset} and \eqref{eq:chainbase}, inequality \eqref{eq:minimizers} (which holds true since $f\leq 0$) can be rewritten as
\[
\begin{split}
\int_\Omega\weakgrad g^2\,\d\mm\leq \int_{\Omega}\weakgrad{(g+f)}^2\,\d\mm=\int_{\{g\leq M-\eps\}}\weakgrad g^2\,\d\mm,
\end{split}
\]
which yields $\int_{\{g> M-\eps\}}\weakgrad g^2\,\d\mm=0$ and, by Theorem \ref{thm:cheeger}, also
\begin{equation}
\label{eq:costante}
\int_{\{g> M-\eps\}}|Dg|^2\,\d\mm=0.
\end{equation}
Let $A$ be the connected component of $\{g> M-\eps\}$ containing $x_0$. By \eqref{eq:perweak} we have $\partial A\subset \Omega$ and 
\begin{equation}
\label{eq:boundary}
g\restr{\partial A}\equiv M-\eps.
\end{equation}
We claim that $g$ is constant in $A$, which clearly is in contradiction with \eqref{eq:boundary}  and $g(x_0)=M$ and thus gives the conclusion.

Given that $g$ is Lipschitz, it is immediate to see that its local Lipschitz constant is an upper gradient. Let $x\in A$ and $r>0$ such that $B_{2r}(x)\subset A$. By the Poincar\'e inequality \eqref{eq:11poinc} applied with $g$ in place of $f$, $|Dg|$ in place of $G$ and using \eqref{eq:costante} we deduce $g\equiv g(x)$ in $B_r(x)$. Hence $g$ is locally constant in $A$ and thus, since $A$ is connected, it is constant. This produces the desired contradiction and gives the proof.
\end{proof}
Combining Proposition   \ref{prop:relmaxmin} and Proposition \ref{thm:maxprinc} we immediately get the following corollary.
\begin{corollary}
\label{cor:weakmax}
Let $(X,\sfd,\mm)$ be an infinitesimally Hilbertian $CD(K,N)$ space, $K\in \R$, $N\in(1,\infty)$. Let $\Omega\subseteq X$ be open and $g:\overline{\Omega}\to\R$ a Lipschitz function in $D(\bd,\Omega)$. 

Assume that  $\bd g\restr\Omega\geq 0$. Then $\max_{\overline{\Omega}} g\leq \max_{\partial\Omega} g$. 
\end{corollary}

\section{Excess estimates}\label{se:exest}
In this section we prove the main result of this paper, namely the Abresch-Gromoll inequality. We will give two versions of it: the first one (Theorem \ref{thm:AG1}) corresponds to the inequality originally proved in \cite{AbreschGromoll90} and the second one (Theorem \ref{thm:AG2}) to the version given by Cheeger-Colding in \cite{Cheeger-Colding96}.

Let 
\[
s_{K,N}(\theta)=
\begin{cases}
\sqrt{\frac{N-1}{K}}\sin\left(\theta\sqrt{\frac{K}{N-1}}\right)&\text{ if $K>0$},\\
\theta&\text{ if $K=0$},\\
\sqrt{\frac{N-1}{-K}}\sinh\left(\theta\sqrt{\frac{-K}{N-1}}\right)&\text{ if $K<0$},
\end{cases}
\]
and notice that the following scaling property holds
\beq
\label{eq:scalings}
s_{K, N}(\lambda\theta)=\lambda s_{\lambda^{2}K, N}(\theta),\qquad \forall \lambda>0.
\eeq
Moreover for any $0<r\leq R$ we define $\phi_{K,N}(r,R)$ as
\[
\phi_{K,N}(r,R)=\int_{r\leq \xi\leq \eta\leq R}\left(\frac{s_{K, N}(\eta)}{s_{K, N}(\xi)}\right)^{N-1}\,\d\eta\,\d\xi.
\]
Notice that $\phi_{K,N}(\cdot,\cdot)$  is smooth, positive and well defined as long as $KR^{2}< (N-1)\pi^{2}$. A straightforward computation gives
\beq
\label{eq:phi'}
\frac{\d}{\d r}\phi_{K,N}(r,R)=-\int_{r}^{R}\left(\frac{s_{K, N}(\eta)}{s_{K, N}(r)}\right)^{N-1}\,\d\eta\leq 0,
\eeq
\beq
\label{eq:phi''}
\frac{\d^{2}}{\d r^{2}}\phi_{K,N}(r,R)=1+(N-1)\frac{s_{K, N}'(r)}{s_{K, N}(r)}\int_{r}^{R}\left(\frac{s_{K, N}(\eta)}{s_{K, N}(r)}\right)^{N-1}\,\d\eta.
\eeq

\begin{lemma}
Let $(X,\sfd,\mm)$ be an infinitesimally Hilbertian $CD(K,N)$ space, $K\in\R$, $N\in (1,+\infty)$, $\bar x\in X$ and $R>0$ such that $KR^{2}< (N-1)\pi^{2}$. 

Then the function   $\phi_{K, N}(\sfd (\bar x, \cdot), R)$ belongs to $D(\bd, B_R(\bar x)\setminus \{ \bar x\})$ and it holds 
\begin{align}
\label{eq:phigeqm}
\bd \phi_{K, N}(\sfd (\bar x, \cdot), R)\restr{B_R(\bar x)\setminus\{\bar x\}}&\geq \mm\restr{B_R(\bar x)\setminus\{\bar x\}},\\
\nonumber
\phi_{K, N}(\sfd (\bar x, \cdot), R)\restr{S_R(\bar x)}&=0.
\end{align}
\end{lemma}

\begin{proof}
By \eqref{eq:lapdis} we know that $\sfd (\bar x, \cdot)\in D(\bd, B_R(\bar x)\setminus \{\bar x\})$.

From $KR^{2}<(N-1)\pi^{2}$ we have $\phi_{K, N}(\cdot, R)\in C^{\infty}((0, R))$ and proposition \ref{prop:chainlap} gives $\phi_{K, N}(\sfd (\bar x, \cdot), R)\in D(\bd, B_R(\bar x)\setminus \{\bar x\})$. Now notice that 
\[
\frac{N\,\tilde\tau_{K,N}(r)-1}{r}=(N-1)\frac{s_{K, N}'(r)}{s_{K, N}(r)},
\]
and that \eqref{eq:cheeger2} implies $\weakgrad \sfd (\bar x, \cdot)=1$ $\mm$-a.e. Using \eqref{eq:phi'}, \eqref{eq:phi''} into \eqref{eq:chainlap}, together with Laplacian comparison \eqref{eq:lapdis} and the previous formula now gives \eqref{eq:phigeqm}. 

The  second property follows from the definition.
\end{proof}

\begin{proposition}
\label{prop:preAB}
Let $(X,\sfd,\mm)$ be an infinitesimally Hilbertian $CD(K,N)$ space, $K\in\R$, $N\in (1,+\infty)$, $\bar x\in X$ and $R>0$ such that $KR^{2}< (N-1)\pi^{2}$. Let also  $g:B_R(\bar x)\to \R$ be a Lipschitz function such that
\begin{itemize}
\item[i)]
$g\geq 0$,
\item[ii)] $g\in D(\bd,B_R(\bar x)\setminus\{\bar x\})$ and 
$\bd g\restr{B_R(\bar x)\setminus\{\bar x\}}\leq a\mm$ for some $a\geq 0$,
\item[iii)]
$g(p)< a\phi_{K, N}(\sfd(\bar x, p), R)$ for some $p$ with $\sfd(\bar x, p)=h<R$.
\end{itemize}
Then
\beq
\label{eq:bound}
g(\bar x)\leq \inf_{0<\theta< h}\Big(\Lip(g)\theta+a\phi_{K, N}(\theta, R)\Big).
\eeq
\end{proposition}

\begin{proof}
To shorten the notation, we set $\phi(x):=\phi_{K, N}(\sfd(\bar x, x), R)$. Suppose \eqref{eq:bound} is false: then for some $\theta\in (0, h)$ it holds
\beq
\label{eq:pl}
g(\bar x)-\Lip(g)\theta>a\phi(\theta, R).
\eeq
Then the trivial inequality $g\restr{S_\theta(\bar x)}\geq g(\bar x)-\Lip(g)\theta$ yields $a\phi-g\leq 0$ on $S_\theta(\bar x)$. Moreover, since by (i) we have $g\geq 0$ and by construction $\phi\restr{S_R(\bar x)}=0$, we also get $a\phi-g\leq 0$ on $S_R(\bar x)$. Thus from $\partial \big(B_R(\bar x)\setminus \overline B_\theta(\bar x)\big)\subseteq S_R(\bar x)\cup S_\theta(\bar x)$ we deduce
\[
(a\phi-g)\restr{\partial (B_R(\bar x)\setminus \overline B_\theta(\bar x))}\leq 0.
\]
Moreover the linearity of the Laplacian, inequality \eqref{eq:phigeqm} and assumption (ii) ensure
\beq
\label{eq:laplpos}
\bd (a\phi-g)\restr{B_R(\bar x)\setminus \overline B_\theta(\bar x)}\geq 0.
\eeq
Hence the  weak maximum principle in Corollary \ref{cor:weakmax} ensures $a\phi-g\leq 0$ on $B_R(\bar x)\setminus \overline B_\theta(\bar x)$. This contradicts assumption (iii), thus the proof is completed.
\end{proof}
We apply this proposition to prove two Abresch--Gromoll type estimates on the excess function in infinitesimally Hilbertian $CD(K, N)$ spaces with $K\leq 0$. Notice that for non-positive $K$  the function $t\mapsto t^{-1}s_{K, N}(t)$ is nondecreasing on $\R^{+}$ and bounded from below by $1$. Therefore it holds
\beq
\label{eq:monotonics}
\frac{\eta}{\xi}\leq \frac{s_{K, N}(\eta)}{s_{K, N}(\xi)}\leq \frac{s_{K, N}(R)}{R}\frac{\eta}{\xi},\qquad \text{for $0<\xi\leq \eta\leq R$},
\eeq
which provides the comparison estimate
\beq
\label{eq:compphi}
\phi_{0, N}(r, R)\leq \phi_{K, N}(r, R)\leq \left(\frac{s_{K, N}(R)}{R}\right)^{N-1}\phi_{0, N}(r, R).
\eeq
Also, the following scaling property is easily verified: 
\beq
\label{eq:scalingphi}
\phi_{K, N}(\lambda r, \lambda R)=\lambda^{2}\phi_{\lambda^{2}K, N}(r, R),\qquad\forall\lambda > 0
\eeq

We now fix some notation. Let $x_0,x_1\in \supp(\mm)$ be two points and $[0,1]\ni t\mapsto \gamma_t\in X$ a  minimizing geodesic joining them. In the following, we will omit the explicit dependence on the points $x_{0}$ and $x_{1}$ with respect to which we compute various quantities. We set 
\[
h(x)=\min_{t\in[0,1]}\sfd (x, \gamma(t)), \qquad l(x)=\min\{\sfd(x_{0},x),\sfd(x_{1},x)\},
\]
and define the excess function as
\[
E(x)=\sfd(x, x_{0})+\sfd(x, x_{1})-\sfd (x_{0}, x_{1}).
\]
By the triangle inequality and basic properties of the distance function it holds 
\beq
\label{eq:basicE}
0\leq E(x)\leq 2h(x), \qquad E\circ\gamma=0, \qquad \Lip(E)\leq 2.
\eeq
We furthermore define
\beq
\label{eq:defck}
c_{K, N}(\theta)=
\begin{cases}
\dfrac{N-1}{\theta}, &\text{if $K=0$},\\
&\\
\sqrt{-K(N-1)}{\rm cotanh}\left(\theta\sqrt{\frac{-K}{N-1}}\right), &\text{if $K<0$},
\end{cases}
\eeq
so that the comparison estimate \eqref{eq:lapdis} can be more compactly written as
\begin{equation}
\label{eq:compdis2}
\sfd(\cdot,\bar x)\in D(\bd,X\setminus \{\bar x\}),\qquad\qquad\bd\sfd(\cdot,\bar x)\restr{X\setminus \{\bar x\}}\leq c_{K,N}\big(\sfd(\cdot,\bar x)\big)\mm.
\end{equation}

\begin{theorem}\label{thm:AG1}
Let $(X,\sfd,\mm)$ be an infinitesimally Hilbertian $CD(K,N)$ space, $K\leq 0$, $N\in (1,+\infty)$. With the same notation as above, let $\bar x\in \supp(\mm)$ be such that $l(\bar x)>h(\bar x)$. Then
\beq
\label{eq:AB}
E(\bar x)\leq 
\begin{cases}
2\frac{N-1}{N-2}\left(D_{K, N}(\bar x)h^{N}(\bar x)\right)^{\frac{1}{N-1}}&\text{if $N>2$},\\
\frac{N}{2-N}D_{K, N}(\bar x)h^{2}(\bar x)&\text{if $1<N<2$},\\
D_{K, N}(\bar x)h^{2}(\bar x)\left(\frac{1}{1+\sqrt{1+D^{2}(\bar x)h^{2}(\bar x)}}+\log\frac{1+\sqrt{1+D^{2}(\bar x)h^{2}(\bar x)}}{D_{K, N}(\bar x) h(\bar x)}\right)&\text{if $N=2$},
\end{cases}
\eeq
where 
\[
D_{K, N}(\bar x)=\left(\frac{s_{K, N}(h(\bar x))}{h(\bar x)}\right)^{N-1}\frac{c_{K, N}(l(\bar x)-h(\bar x))}{N}.
\]
\end{theorem}

\begin{proof}
Let $h(\bar x)<r<l(\bar x)$ and notice that by the monotonicity of $c_{K, N}$ we get
\[
c_{K, N}(\sfd (x_{i}, x))\leq c_{K, N}(l(\bar x)-r),\qquad \forall x\in B_r(\bar x), \ \ i=1,2.
\]
Estimate \eqref{eq:compdis2} and the linearity of the Laplacian thus ensure
\[
\bd E\restr{B_r(\bar x) }\leq 2c_{K, N}(l(\bar x)-r)\mm.
\]
Now let $p$ be a point on $\gamma$ such that $\sfd(\bar x, p)=h(\bar x)$.  By the positivity of $\phi_{K, N}$ we have
\[
0=E(p)< 2c_{K, N}(l(\bar x)-r)\phi_{K, N}(\sfd(\bar x, p), R),\qquad \forall R>h.
\]
Hence we  can apply Proposition \ref{prop:preAB}, and taking $R\downarrow h(\bar x)$ in \eqref{eq:bound} we obtain 
\[
E(\bar x)\leq \inf_{0<\theta<h}\Big(2\theta+2c_{K, N}(l(\bar x)-h(\bar x))\phi_{K, N}(\theta, h(\bar x))\Big).
\]

To estimate such a minimum we make use of \eqref{eq:compphi} and the explicit form of $\phi_{0, N}$ to get, for $N\neq 2$, $N>1$
\[
2\theta+2c_{K, N}(l(\bar x)-h(\bar x))\phi_{K, N}(\theta, h(\bar x))\leq 2\theta+
D_{K, N}(\bar x)\big(\theta^{2}-\frac{N}{N-2}h^{2}(\bar x)+\frac{2}{N-2}h^{N}(\bar x)\theta^{2-N}\big).
\]
The right hand side is a convex function $\Phi(\theta)$ such that $\Phi'\to-\infty$ for $\theta\downarrow 0$ and $\Phi'(h(\bar x))=2>0$, therefore its minimum is achieved at the unique point $\theta_{0}\in (0, h(\bar x))$ satisfying $\Phi'(\theta_{0})=0$. It is easily seen that $\theta_{0}$ satisfies
\beq
\label{eq:thetazero}
\theta_{0}^{N-1}=D_{K, N}(\bar x)(h^{N}(\bar x)-\theta_{0}^{N})\leq D_{K, N}(\bar x)h^{N}(\bar x),
\eeq
and some algebra gives
\beq
\label{eq:eestim}
E(\bar x)\leq 2\frac{N-1}{N-2}\theta_{0}+\frac{D_{K, N}(\bar x)}{N-2}(\theta^{2}_{0}-h^{2}(\bar x)).
\eeq
Now if $N>2$ the second term is negative, and using \eqref{eq:thetazero} we obtain the first inequality in \eqref{eq:AB}. 
Similarly, if $1<N<2$, the first term in \eqref{eq:eestim} is negative and we obtain the second inequality. For $N=2$ the integral gives rise to a logarithmic factor, and an explicit calculation along the same lines as above gives the last estimate in \eqref{eq:AB}.
\end{proof}

Observe that, by the definition \eqref{eq:defck} of $c_{K, N}$, it holds
\[
\lim_{K, l^{-1}\to 0} c_{K, N}(l-h(\bar x))=0.
\]
Therefore, for bounded $h(\bar x)$, $D_{K, N}(\bar x)\to 0$ when $l(\bar x)\uparrow+\infty$ and $K\uparrow 0$. Let then $p\neq x_0, x_1$ a point on the minimizing geodetic $\gamma$ and $R>0$ such that $2R<l(p)$. By the triangle inequality, for any $\bar x\in B_R(p)$ it holds $h(\bar x)\leq R$, $l(p)-R< l(\bar x)$ and \eqref{eq:AB} implies that for suitable $\alpha_{N}>0$  it holds
\[
\sup_{B_R(p)}E\leq C(N, R) \left(c_{K, N}(l(p)-2R)\right)^{\alpha_{N}}\to 0\qquad \text{for $l(p)\uparrow +\infty$ and $K\uparrow 0$.}
\]
This asympthotic behaviour relies on the fact that $E(p)=0$, and the next Abresch--Gromoll--type excess estimate deals with the case $E(p)\neq 0$. In order to deal with scale-invariant quantities, we define
\[
\Psi_{N, R}(E, l, K)=
\begin{cases}
E+\sqrt{ER}+(ER^{N-2})^{\frac{1}{N-1}}+(c_{K, N}(l-3R)R^{N})^{\frac{1}{N-1}}&\text{if $N>2$},\\
E\log(2+\frac{R}{E})+\sqrt{ER}+c_{K, N}(l-3R)R^{2}\log\left(2+\frac{1}{c_{K, N}(l-3R)R}\right)&\text{if $N=2$},\\
E+\sqrt{ER}+c_{K, N}(l-3R)R^{2}&\text{if $N<2$}.
\end{cases}
\]
and notice that
\[
\lim_{E, l^{-1}, K\to 0} \Psi_{N, R}(E, l, K)=0.
\]

\begin{theorem}\label{thm:AG2}
Let $(X,\sfd,\mm)$ be an infinitesimally Hilbertian $CD(K,N)$ space, $K\leq 0$,  $N\in (1,+\infty)$. 

Then with the notation as above for any $\alpha>0$ there exist $A(N, \alpha),C(N,\alpha)>0$ such that the following is true: given any $p\in\supp(\mm)$ and $R>0$  such that 
\begin{itemize}
\item[i)]
$l(p)>3R$,
\item[ii)]
$ KR^{2}\geq -\alpha$, 
\item[iii)] 
$2R\, c_{K, N}(l(p)-3R)\leq A(N, \alpha)$,
\end{itemize}
it holds
\beq
\label{eq:ABestimate2}
\sup_{B_R(p)}E\leq C(N, \alpha)\Psi_{N, R}(E(p), l(p), K).
\eeq
\end{theorem}

\begin{proof}
We start with the following

\noindent{\bf Claim: } {\em if $g$ satisfies the hypotheses of Proposition \ref{prop:preAB} in $B_{2r}(\bar x)$ for some $\bar x$ such that $\sfd(\bar x, p)=r\leq R$ and an $a$ such that
\beq
\label{eq:conda}
a\leq \frac{N}{2^{N}R}\left(\frac{s_{K, N} (2R)}{2R}\right)^{1-N},
\eeq
then $g(\bar x)$ is bounded as in \eqref{eq:AB} with the following choices:}
\[
h:=2r\qquad \text{and} \qquad D_{K, N}(a, R):=\left(\frac{s_{K, N}(2R)}{2R}\right)^{N-1}\frac{a}{2N}.
\]
Indeed, for $r\leq R$ and $a>0$ the function
\[
\Phi_{r}(\theta)=2\theta +\frac{a}{2N}\left(\frac{s_{K, N}(2R)}{2R}\right)^{N-1}\phi_{0, N}(\theta, 2r).
\]
is convex and $\Phi_{r}'\downarrow -\infty$ for $\theta\downarrow 0$. Moreover, if  \eqref{eq:conda} holds,  it is easily verified that $\Phi_{r}'(r)\geq 0$ and therefore its minimum
is achieved at some $\theta_0\in (0, r)$. Now the same computations done in  the previous theorem give the claim.

Let then 
\[
A(N, \alpha):=\frac{N}{2^{N}}\left(s_{-4\alpha, N}(1)\right)^{1-N},
\]
and notice that using \eqref{eq:scalings}, \eqref{eq:monotonics} and (ii) we get the bound 
\beq
\label{eq:A}
A(N, \alpha)\leq  \frac{N}{2^{N}}\left(\frac{s_{K, N} (2R)}{2R}\right)^{1-N}.
\eeq
Since for any $\bar x\in B_R(p)$, $B_{2R}(\bar x)\subseteq B_{3R}(p)$,  (i) and the monotonicity of $c_{K, N}$ imply that 
\beq
\label{eq:laplc}
\text{for any $\bar x$ such that $\sfd(\bar x, p)=r\leq R$},\qquad  \bd E\restr{B_{2r}(\bar x)}\leq 2c_{K, N}(l(p)-3R),
\eeq
which guarantees the validity of hypothesis (ii) in $B_{2r}(\bar x)$ (and, clearly, (i)) in Proposition \ref{prop:preAB} for any $\bar x\in B_R(p)$ and any $a\geq 2c_{K, N}(l(p)-3R)$. Letting now 
\[
\bar E(r):=\sup_{\bar x\in S_r(p)}E(\bar x),
\]
we will estimate $\bar E(r)$ separately for three ranges of $r$, depending on how large $\phi_{K, N}(r, 2r)$ is compared to $E(p)$. In each range we will estimate $\bar E(r)$ through a constant multiple (depending only on $\alpha$ and $N$) of a term appearing in $\Psi_{K, N}$, and summing up the bounds we will obtain \eqref{eq:ABestimate2}.  We consider for simplicity just the case $N>2$, the other ones following along the same lines.

\noindent{\bf Case 1: } suppose $r$ satisfies 
\[
E(p)<2c_{K, N}(l(p)-3R)\phi_{K, N}(r, 2r).
\] 

By \eqref{eq:laplc} the hypotheses of Proposition \ref{prop:preAB} are satisfied for $a=2c_{K, N}(l(p)-3R)$. Moreover,  (iii) and \eqref{eq:A} guarantee that \eqref{eq:conda} holds true for this value of $a$, and the initial claim provides the estimate
\beq
\label{eq:E1}
\bar E(r)\leq 2\frac{N-1}{N-2}\left(D_{K, N}(a, R)(2r)^{N}\right)^{\frac{1}{N-1}}\leq2\frac{N-1}{N-2} \left(\frac{c_{K, N}(l(p)-3R)R^{N}}{A(N, \alpha)}\right)^{\frac{1}{N-1}}.
\eeq
where we used the inequality
\[
2^{N}D_{K, N}(a, R)\leq \frac{a}{2 A(N, \alpha)},
\]
which holds by \eqref{eq:A}.

\noindent{\bf Case 2: } suppose $r$ satisfies 
\[
2c_{K, N}(l(p)-3R)\phi_{K, N}(r, 2r)\leq E(p)< \frac{A(N, \alpha)}{R}\phi_{K, N}(r, 2r),
\]
(notice that by (iii) the left hand side is not greater the right one), and choose $\eps>0$ such that the previous chain of inequalities holds true with $E(p)+\eps$ instead of $E(p)$.
Recalling \eqref{eq:laplc}, the hypotheses of Proposition \ref{prop:preAB} hold true for 
\[
a=\frac{E(p)+\eps}{\phi_{K, N}(r, 2r)}\geq 2c_{K, N}(l(p)-3R).
\]
Moreover, by \eqref{eq:A} and the condition on $r$, $a$ also satifies \eqref{eq:conda}. Applying the claim and letting $\eps\downarrow 0$ gives, for these $r$'s, the bound
\beq
\label{eq:basta}
\bar E(r)\leq  2\frac{N-1}{N-2}\left(\frac{E(p)r^{N}}{A(N, \alpha)\phi_{K, N}(r, 2r)}\right)^{\frac{1}{N-1}}.
\eeq
Using the scaling property \eqref{eq:scalingphi} and estimate \eqref{eq:compphi}, one gets
\beq
\label{eq:phir2r}
\phi_{K, N}(r, 2r)= r^{2}\phi_{r^{2}K, N}(1, 2)\geq r^{2}\phi_{0, N}(1,2),
\eeq
which, plugged into \eqref{eq:basta}, gives
\beq
\label{eq:E2}
\bar E(r)\leq 2\frac{N-1}{N-2}\left(\frac{E(p)R^{N-2}}{A(N, \alpha)\phi_{0, N}(1, 2)}\right)^{\frac{1}{N-1}}.
\eeq

\noindent{\bf Case 3: } suppose finally that $r$ satisfies 
\[
\frac{A(N, \alpha)}{R}\phi_{K, N}(r, 2r)\leq E(p).
\] 

From \eqref{eq:phir2r} we get 
\[
r\leq \sqrt{\frac{E(p)R}{A(N, \alpha)\phi_{0, N}(1,2)}},
\]
and since $E$ is 2-Lipschitz we obtain
\beq
\label{eq:E3}
\bar E(r)\leq E(p)+2r\leq E(p)+2\sqrt{\frac{E(p)R}{A(N, \alpha)\phi_{0, N}(1,2)}}.
\eeq
We conclude summing up the right hand sides of \eqref{eq:E1}, \eqref{eq:E2}, \eqref{eq:E3} to obtain a bound valid for all $r\leq R$, which is easily checked to be of the form \eqref{eq:ABestimate2}. 
\end{proof}

\bibliographystyle{siam}
\bibliography{biblio}

\end{document}